\documentclass[12pt]{article}
\usepackage[utf8]{inputenc}
\usepackage{amssymb}
\usepackage{amsmath}
\usepackage{amsthm}
\usepackage{hyperref}
\usepackage{a4wide}
\usepackage{breqn}
\usepackage[table,xcdraw]{xcolor}
\usepackage{xcolor}
\usepackage{mathtools}
\usepackage{bm}
\usepackage{multirow}
\usepackage{enumitem}
\usepackage{mathrsfs}
\usepackage{graphicx}
\usepackage{caption}
\usepackage{url}
\usepackage{placeins}
\usepackage[english]{babel}
\usepackage{array}
\usepackage{subcaption}
\usepackage{booktabs}
\usepackage[numbers]{natbib}
\usepackage{lineno}

\usepackage{graphics}
\usepackage{lineno} 
\usepackage{graphicx}
\numberwithin{equation}{section}

\newtheorem{theorem}{\bf Theorem}[section]
\newtheorem{definition}{Definition}[section]
\newtheorem{corollary}{Corollary}[section]

\newtheorem{lemma}{Lemma}[section]
\newtheorem{remark}{Remark}[section]
\theoremstyle{remark}
\newtheorem{exam}{\bf Example}[section]

\def \vec{\mathrm v\mathrm e \mathrm c}
\def \R{{\mathbb R}}
\def \F{{\mathsf{F}}}
\def \C{{\mathbb{C}}}
\def \I{{\mathfrak{I}}}

\def \x{{\bm x}}

\def \T{\mathsf T}
\def\bmatrix#1{\left[\begin{matrix}
		#1
	\end{matrix}\right]}

\def \diag{\mathrm{diag}}

\def \D{{\Delta}}
\def \r{{\bf r}}

\def  \0{{0}}

\def \u{\bm u}
\def \p{\bm p}

\def \T{\mathsf T}

\def\bmatrix#1{\left[ \begin{matrix} #1 \end{matrix} \right]}

\def \RR{\mathfrak{R}}

\def \R{{\mathbb R}}

\title{Structured Backward Errors of Sparse Generalized Saddle Point Problems with Hermitian Block Matrices}
\author{Sk. Safique Ahmad\footnotemark[2] \footnotemark[1] \and Pinki Khatun \footnotemark[2] }
\date{}
\begin{document}
\maketitle	
\begin{abstract}
    In this paper, we derive the structured backward error (BE) for a class of generalized saddle point problems (GSPP) by preserving the sparsity pattern and Hermitian structures of the block matrices. Additionally, we construct the optimal backward perturbation matrices for which the structured BE is achieved. Our analysis also examines the structured BE in cases where the sparsity pattern is not maintained. Through numerical experiments, we demonstrate the reliability of the derived structured BEs and the corresponding optimal backward perturbations.  Additionally, the derived structured BEs are used to assess the strong backward stability of numerical methods for solving the GSPP.
\end{abstract}
	
   
		
		\footnotetext[1]{Corresponding author.}
	\footnotetext[2]{
		Department of Mathematics, Indian Institute of Technology Indore, Simrol, Indore 453552, Madhya Pradesh, India. \texttt{Email}: \texttt{safique@iiti.ac.in} (Sk. Safique Ahmad),   \texttt{pinki996.pk@gmail.com} (Pinki Khatun)}

\section{Introduction}
The concept of backward error (BE) analysis was proposed by \citet{Wilkinson1965},  plays a crucial role in the field of numerical linear algebra. It has several key applications: for example, it can be used to determine a nearly perturbed problem with minimal norm perturbation, ensuring that the approximate/computed solution of the original problem aligns with the exact solution of the perturbed problem; by taking the product of the condition numbers and the BE, an upper bound on the forward error can be established; BEs are often employed as a stopping criterion for iterative algorithms when solving a problem.  
For a given problem, if the computed BE of an approximate solution is within the unit round-off error, then the corresponding numerical algorithm is considered backward stable \cite{higham2002}. The notion of structured BE is introduced when the problem processes some special structure, and the BE is analyzed with structure-preserving constraints imposed on the perturbation matrices.  Furthermore, a numerical algorithm is classified as strongly backward stable if the computed structured BE remains within the unit round-off error \cite{strongweak, strongstab}.

This paper considers the $2\times 2$  block  system of linear equations of the following  form:
\begin{align}\label{eq11}
    \mathfrak{B}{\bm x}\triangleq \bmatrix{E & F^{*} \\ H &G}\bmatrix{\bm{u} \\ \bm{p}}=\bmatrix{q\\r}\triangleq {\bm f},
\end{align}
where $E\in \C^{n\times n}, F,H\in \C^{m\times n}, G\in \C^{m\times m},$ $q\in \C^n$ and $r\in \C^m.$ Hereafter, $B^{*}$ represents for conjugate transpose of $B.$ 
The \(2 \times 2\) block linear system encompasses several important cases: the Hermitian saddle point problem (SPP) (\(E = E^*\), \(F = H\), \(G = 0\)), the non-Hermitian SPP (\(F = H\), \(G = G^*\)), and the real standard SPP (\(E \in \R^{n \times n},\, F, H \in \R^{m \times n},\, G \in \R^{m \times m},\, q \in \R^n,\) and \(r \in \R^m\)) \cite{ZZBai2021, Benzi2005}. Here, $0$ denotes the zero matrix of appropriate size. We refer \eqref{eq11} as the generalized SPP (GSPP). The GSPP of the form\eqref{eq11}  has broad applications across various scientific and engineering domains, including computational fluid dynamics \cite{navier2015, Elman2005}, optimal control \cite{OPTIMAL}, weighted and equality constrained least-squares estimation \cite{LSproblem2002}, and so on. For fundamental properties, as well as a comprehensive survey and applications of GSPP, refer to \cite{Benzi2005}.

The development of various iterative methods for solving the GSPP \eqref{eq11} has become a focal point for many researchers, as reflected in recent studies \cite{ Uzawa-DOS2024, CSPP2009,  CSPP2022, CSPP2021} and reference therein.  However, the computed solution may still contain some errors and can potentially lead to insignificant results. Therefore, it is crucial to assess how closely the computed solution approximates the solution to the original problem. 

Recently, many studies have been carried out on structured BE analysis and condition numbers for real SPPs \cite{ PinkiGSPP, LAA_Pinki, be2012LAA, be2017ma, be2022lma, Sun1999, be2007wei,  BE2020BING}. {The existing literature primarily considers the block matrices $E, F, G,$ and $ H$ real, and $(\cdot)^*$ in \eqref{eq11} is  considered as real transpose (denoted by $(\cdot)^{\T}$). A brief overview of literature work is as follows:  by considering $E=E^{\T}\in \R^{n\times n}, F=H\in \R^{m\times n}$ and $G=0,$ \citet{Sun1999} derived the structured BE for the GSPP. By employing Sun's methods, \citet{BEKKT2004} investigated the  BE for the GSPP when $E=I_n\in \R^{n\times n}, F=H\in \R^{m\times n}$ and $G=0\in \R^{m\times m}$ and  \citet{be2007wei} study the structured BE for the case $E\neq E^{\T}\in \R^{n\times n}, F=H\in \R^{m\times n}$ and $G=0\in \R^{m\times m}.$ Here, $I_n$ denotes the $n\times n$  identity matrix.
 Further,  when $G\neq 0$, structured BEs for the GSPP have been studied under the following matrix structures: 
 \begin{enumerate}
     \item[(a)] $E=E^{\T}\in \R^{n\times n}, F=H\in \R^{m\times n}$  and $G=G^{\T} \in \R^{m\times m}$ in \cite{be2012LAA, be2022lma, BE2020BING};
     \item[(b)] $E=E^{\T}\in \R^{n\times n},$ $F= H\in \R^{m\times n}$  and $G\in \R^{m\times m}$ in \cite{BE2020BING};
     \item[(c)] $E\in \R^{n\times n},$ $F= H\in \R^{m\times n}$  and $G\in \R^{m\times m}$ in \cite{be2012LAA, be2017ma};
     \item[(d)] $E=E^{\T}\in \R^{n\times n},$ $F\neq H\in \R^{m\times n}$  and $G\in \R^{m\times m}$ in \cite{be2012LAA, be2022lma};
     \item[(e)] $E= E^{\T}\in \R^{n\times n},$ $F\neq H\in \R^{m\times n}$  and $G=G^{\T}\in \R^{m\times m}$ in \cite{BE2020BING};
     \item[(f)] $E\in \R^{n\times n},$ $F\neq H\in \R^{m\times n}$  and $G=G^{\T}\in \R^{m\times m}$ in \cite{be2012LAA}.
 \end{enumerate}
Recent studies have also focused on structured BE for SPP with three-by-three block structures \cite{threeBE2023, BE2024}. These investigations consider cases where the block matrices are real, with diagonal blocks being either symmetric or nonsymmetric. 
Nevertheless, the existing literature reveals several shortcomings: $(1)$ they do not explore structured BE analysis when the block matrices in \eqref{eq11} are complex, specifically when block matrices possess Hermitian structure; 
$(2)$ they do not provide the optimal backward perturbations needed to achieve the structured BE.} 

In many practical applications, such as the discretization of the Stokes equation \cite{Elman2005} and PDE-constrained optimization problems,  the coefficient matrix of the GSPP includes large numbers of zeros, i.e., prosesses sparsity structure. Preserving this sparsity is crucial for computational efficiency and maintaining the problem's structure. Recent works on structured BE analysis for eigenvalue problems \cite{prince2020, prince2021, PEVP2012} have highlighted the importance of incorporating sparsity preservation in perturbation analysis. Therefore, performing perturbation analysis that preserves the sparsity pattern is essential, which requires the construction of optimal sparse perturbation matrices to ensure accuracy and efficiency in solving the GSPP. The existing literature on structured BE for GSPP also does not preserve the sparsity of the coefficient matrix.

To overcome these drawbacks, in this paper, by preserving the sparsity pattern of $\mathfrak{B},$ we investigate the structured BEs in the following cases:
\begin{align}
 & (i)~  E\in \mathbb{HC}^{n\times n}, ~F=H\in \C^{m\times n},~ G\in \C^{m\times m},\\
 & (ii)~ E\in \mathbb{C}^{n\times n},~ F=H\in \C^{m\times n},~ G\in \mathbb{HC}^{m\times m},\\
 & (iii)~ E\in \mathbb{HC}^{n\times n},~ F\neq H\in \C^{m\times n},~ G\in \mathbb{HC}^{m\times m},
\end{align}
where $\mathbb{HC}^{n\times n}$ represents the collection of all Hermitian matrices.

The main highlights of this paper are listed as follows:
\begin{itemize}
    \item We investigate the structured BE for the GSPP \eqref{eq11} for  cases $(i)$-$(iii)$ by retaining the the sparsity pattern of $\mathfrak{B}.$
    \item We provide compact formulae for the structured BE in each of the three cases and derive the formulae for the optimal backward perturbations. Moreover, we derive the structured BEs for the GSPP when the sparsity of $\mathfrak{B}$ is not considered.
    \item We present numerical examples to validate the derived results and evaluate the strong backward stability of numerical methods for solving the GSPP \eqref{eq11}.
\end{itemize}

The remainder of the paper is organized as follows: Section \ref{Natation} presents essential basic notations, preliminaries, and a few important results. In Section \ref{sec:SBE}, we derive the structured BEs for cases (\textit{i}), (\textit{ii}), and (\textit{iii}). Section \ref{sec:numerical} provides numerical examples to validate the obtained formulae for structured BEs, and concluding remarks are presented in Section \ref{sec:conclusion}.

\section{Notation and Preliminaries}\label{Natation}
\subsection{Notation}
Throughout the paper, let $\mathbb{R}^{m \times n}$ and $\C^{m\times n}$ denote the set of all  $m \times n$ real and complex matrices, respectively. We use $\mathbb{SR}^{n\times n},$  $\mathbb{SKR}^{n\times n},$ and $\mathbb{HC}^{n\times n}$ represent the set of all $n \times n$ real symmetric matrices, real skew-symmetric matrices, and  Hermitian matrices, respectively. For $X\in \C^{n\times m},$ $\mathfrak{R}(X)$, $\mathfrak{I}(X),$ $X^{\T},$ $X^{*},$ and $X^{\dagger}$ represent the real part, imaginary part, transpose, conjugate transpose, and Moore-Penrose inverse of $X,$ respectively. We use $0_{m\times n}$ and $I_m$ to denote the $m\times n$  zero matrix and identity matrix, respectively (we only use 0 when the size is clear). We use $e_i^m$ to denote $i^{th}$ column of the identity matrix $I_m$. The symbol ${\bf 1}_{m\times m}$ represents $m\times m$ matrix with all entries equal to $1$. The notations \( \|\cdot\|_2 \) and \( \|\cdot\|_{\mathsf{F}} \) represent the Euclidean norm and the Frobenius norm, respectively. The componentwise multiplication of the matrices of $X, Z\in \R^{m\times n}$ is defined as $X\odot Z=[x_{ij}z_{ij}]\in \R^{m\times n}.$ For  $Z\in \R^{m\times n}$, we define $\Theta_X:={\tt sgn}(Z)=[{\tt sgn}(z_{ij})]\in \R^{m\times n},$ where
 \begin{align*}
     {\tt sgn}(z_{ij})=\left\{ \begin{array}{lcl}
     1, &\text{for}~z_{ij}\neq 0,\\
     0,  &\text{for}~ z_{ij}=0.
      \end{array}\right.
 \end{align*}
 For  $X=[\x_1, \x_2,\ldots, \x_n]\in \R^{m\times n},$ where $\x_i\in \R^{m},$ $i=1,2,\ldots, n,$  we define $\vec(X):=[\x_1^{\T}, \x_2^{\T},\ldots, \x_n^{\T}]^{\T}\in\R^{mn}$. For the matrices $X\in \R^{m\times n}$ and $Y\in \R^{p\times q},$  the Kronecker product \cite{kronecker1981} is defined as $X\otimes Y= [x_{ij}Y]\in \R^{mp\times nq}.$ For the matrix $A\in\R^{m\times n}$  and vectors $u\in \R^{n}$ and $v\in \R^{m},$  the following properties of the Kronecker product and the vec operator hold \cite{kronecker1981, kronecker2004}:
 \begin{align}\label{kron}
\left\{ \begin{array}{c}
Au =  (u^{\bm\top}\otimes I_m)\vec(A), \\
  A^{\T}v= (I_n\otimes v^{\T})\vec(A).\\
 \end{array}\right.
  \end{align}
Given  positive weight vectors ${\sigma}_1=[\alpha_1,\alpha_2,\alpha_3, \beta_1, \beta_2]^{\T}$ and ${\sigma}_2=[\alpha_1,\alpha_2,\alpha_3,\alpha_4, \beta_1, \beta_2]^{\T},$ then we define the weighted Frobenius norms are defined as follows:
\begin{align*}
&\bm{\zeta}^{{\sigma}_1}(E,F,G,q,r)=\sqrt{\alpha_1^2\|E\|^2_{\F}+\alpha_2^2\|F\|^2_{\F}+\alpha_3^2\|G\|^2_{\F}+\beta_1^2\|q\|^2_{2}+\beta_2^2\|r\|^2_{2}},\\
&\bm{\zeta}^{{\sigma}_2}(E,F, H,G,q,r)=\sqrt{\alpha_1^2\|E\|^2_{\F}+\alpha_2^2\|F\|^2_{\F}+\alpha_3^2\|H\|^2_{\F}+\alpha_4^2\|G\|^2_{\F}+\beta_1^2\|q\|^2_{2}+\beta_2^2\|r\|^2_{2}}.
    \end{align*}

\subsection{Preliminaries}
In this subsection, we first present the notion of unstructured and structured BE and a few important definitions and lemmas.
\begin{definition}\label{def:UBE}\cite{Rigal1967}
     Let $\widehat{{\bm x}}=[\widehat{\u}^{\T},\widehat{\p}^{\T}]^{\T}$ be an approximate solution of the GSPP \eqref{eq11}.Then, the normwise unstructured BE, denoted by ${\bm \xi}(\tilde{{\bm x}}),$ is defined as: 
 \begin{eqnarray}
 \nonumber
{\bm \xi}({\widehat{{\bm x}}})&:=&\min_{(\Delta \mathfrak{B},~\Delta {\bm f})\in \mathcal{F}}\left\{\left\|\bmatrix{\frac{\|\Delta \mathfrak{B}\|_{\F}}{\|\mathfrak{B}\|_{\F}} & \frac{\|\Delta {\bm f}\|_{2}}{\|{\bm f}\|_{2}}}\right\|_2\right\}, 
\end{eqnarray}
where \begin{eqnarray}
    \mathcal{F}=\left\{(\Delta \mathfrak{B},~\Delta {\bm f}) {\big |}( \mathfrak{B}+\Delta \mathfrak{B})\widehat{{\bm x}}={\bm f}+\Delta {\bm f}\right\}.
\end{eqnarray}
\end{definition}
\citet{Rigal1967} provided explicit expression for the BE defined in \eqref{def:UBE} which is given by
\begin{equation}\label{UBE:exp}
    {\bm \xi}({\widehat{{\bm x}}})=\frac{\|{\bm f}-\mathfrak{B} \widehat{{\bm x}}\|_2}{\sqrt{\|\mathfrak{B}\|^2_{\mathsf{F}}\|\widehat{{\bm x}}\|_2^2+\|{\bm f}\|_2^2}}.
\end{equation}
 When ${\bm \xi}(\widehat{\bm x})$ is sufficiently small, the approximate solution $\widehat{\bm x}$ becomes the exact solution to a slightly perturbed system $(\mathfrak{B} + \Delta \mathfrak{B})\widehat{\bm x } = {\bm f} + \Delta {\bm f}$, where both $\|\Delta \mathfrak{B}\|_F$ and $\|\Delta {\bm f}\|_2$ are relatively small. This implies that the corresponding numerical algorithm exhibits backward stability.

In the following definition, we introduce the concept of structured BE for the GSPP \eqref{eq11}. Throughout the paper, we assume that the coefficient matrix $\mathfrak{B}$ in \eqref{eq11} is nonsingular.
\begin{definition}\label{def:SBE}
   Assume that $\widehat{\bm{x}}=[\widehat{\bm{u}}^{\T},\widehat{\bm{p}}^{\T}]^{\T}$ is a computed solution of the GSPP \eqref{eq11}. Then, we define the  normwise structured BEs $\bm{\xi}^{\mathcal{G}_i}(\widehat{\bm{u}},\widehat{\bm{p}}),$ $i=1,2,3,$ as follows:
 \begin{align*}
&\bm{\xi}^{\mathcal{G}_i}(\widehat{\bm{u}},\widehat{\bm{p}})=\displaystyle{\min_{\left(\begin{array}{c}
       \D E,\D F,\\
     \D G, \D q, \D r 
     \end{array}\right)\in\, \mathcal{G}_i}} {\bm{\zeta}^{{\sigma}_1}(\D E,\D F,\D G, \D q, \D r)},
     \quad \mbox{for}\, \, i=1,2,\\
&\bm{\xi}^{\mathcal{G}_3}(\widehat{\bm{u}},\widehat{\bm{p}})=\displaystyle{\min_{\left(\begin{array}{c}
       \D E,\D F, \D H,\\
     \D G, \D q, \D r 
     \end{array}\right)\in\, \mathcal{G}_3}} {\bm{\zeta}^{{\sigma}_2}(\D E,\D F, \D H, \D G, \D q, \D r)},
     \end{align*}

 where 
 \begin{align}\label{s3:eq24}
     &\nonumber\mathcal{G}_{1}=\Bigg\{\left(\begin{array}{c}
       \D E,\D F,\\
     \D G, \D q, \D r 
     \end{array}\right) {\bigg |} \bmatrix{E+\D E & (F+\D F)^{*}\\ F+\D F & G+\D G}\bmatrix{\widehat{\bm{u}}\\ \widehat{\bm{p}}}=\bmatrix{q+\D q\\ r+\D r},\\
     &\hspace{1.5cm}\D E \in\mathbb{HC}^{n\times n},\D F\in \C^{m\times n}, \D G\in \C^{m\times m},\, \D q\in \C^{n},\D r\in \C^{m}\Bigg\},\\ \nonumber
     & \mathcal{G}_2=\Bigg\{\left(\begin{array}{c}
       \D E,\D F,\\
     \D G, \D q, \D r 
     \end{array}\right) {\bigg |} \bmatrix{E+\D E & (F+\D F)^{*}\\ F+\D F & G+\D G}\bmatrix{\widehat{\bm{u}}\\ \widehat{\bm{p}}}=\bmatrix{q+\D q\\ r+\D r},\\ \label{s3:eq25}
     &\hspace{1.5cm}\D E \in \C^{n\times n},\D F\in \C^{m\times n}, \D G\in \mathbb{HC}^{m\times m},\, \D q\in \C^{n},\D r\in \C^{m}\Bigg\},
     \\ \nonumber
     & \mathcal{G}_3=\Bigg\{\left(\begin{array}{c}
       \D E,\D F, \D H,\\
     \D G, \D q, \D r 
     \end{array}\right) {\bigg |} \bmatrix{E+\D E & (F+\D F)^{*}\\ H+\D H & G+\D G}\bmatrix{\widehat{\bm{u}}\\ \widehat{\bm{p}}}=\bmatrix{q+\D q\\ r+\D r},\\ \label{s3:eq26}
     &\hspace{1.2cm}\D E \in  \mathbb{HC}^{n\times n},\D F, \D H\in \C^{m\times n}, \D G\in \mathbb{HC}^{m\times m},\, \D q\in \C^{n},\D r\in \C^{m}\Bigg\}.
  \end{align}
\end{definition}       

By choosing $\alpha_1=\frac{1}{\|E\|_{\mathsf{F}}},$ $\alpha_2=\frac{1}{\|F\|_{\mathsf{F}}},$ $\alpha_3=\frac{1}{\|G\|_{\mathsf{F}}},$ (or $\alpha_3=\frac{1}{\|H\|_{\mathsf{F}}},$ $\alpha_4=\frac{1}{\|G\|_{\mathsf{F}}}$ for $\bm{\xi}^{\mathcal{G}_3}(\widehat{\u}, \widehat{p})$), $\beta_1=\frac{1}{\|q\|_{2}},$ and  $\beta_2=\frac{1}{\|r\|_{2}},$ we obtain relative structured BEs for the GSPP \eqref{eq11}.

\begin{remark}
   We denote the optimal backward perturbations for the structured BEs by $\D E_{\tt{opt}},$ $\D F_{\tt{opt}},$ $\D E_{\tt{opt}},$  $\D G_{\tt{opt}},$  $\D q_{\tt{opt}},$ $\D r_{\tt{opt}}.$ Therefore, the following holds:
    \begin{align*}
        &\bm{\xi}^{\mathcal{G}_i}(\widehat{\bm{u}},\widehat{\bm{p}})={\bm{\zeta}^{{\sigma}_1}(\D E_{\tt{opt}},\D F_{\tt{opt}},\D G_{\tt{opt}}, \D q_{\tt{opt}}, \D r_{\tt{opt}})}~~\text{for}~~i =1,2,\\
       & \bm{\xi}^{\mathcal{G}_3}(\widehat{\bm{u}},\widehat{\bm{p}})={\bm{\zeta}^{{\sigma}_2}(\D E_{\tt{opt}},\D F_{\tt{opt}},\D H_{\tt{opt}},\D G_{\tt{opt}}, \D q_{\tt{opt}}, \D r_{\tt{opt}})}.
    \end{align*}
    
\end{remark}
\begin{remark}
Our focus is on examining the structured BE while preserving the sparsity pattern of the block matrices. To accomplish this, we substitute the perturbation matrices $\D E$, $\D F$, $\D H$, and $\D G$ with $\D E \odot \Theta_E$, $\D F \odot \Theta_F$, $\D H \odot \Theta_H$, and $\D G \odot \Theta_G$, respectively. In this framework, the structured BEs are denoted as ${\bm\xi}_{\bf sps}^{\mathcal{G}_i}(\u, \p)$ for $i = 1, 2, 3$. Additionally, we represent the optimal perturbation matrices by ${\D E}^{\tt sps}_{\tt opt}$, ${\D F}^{\tt sps}_{\tt opt}, {\D H}^{\tt sps}_{\tt opt}$, ${\D G}^{\tt sps}_{\tt opt},$ ${\D q}_{\tt opt}$ and ${\D r}_{\tt opt}.$
\end{remark}

Next, we discuss some important definitions and lemmas.

\begin{lemma}\label{sec2:lemma}\cite{DSWATKINS}
 Consider the system of linear equations $Ax = b,$ where $A\in \R^{n\times m}, b\in \R^n$ is consistent if and only if $AA^\dagger  b = b$. Furthermore, when the system is consistent, the solution with the minimum norm is given by $A^\dagger b$.
\end{lemma}
\begin{definition}
    Let  $Z\in \mathbb{SR}^{m\times m},$ then we define its generator vector by $$\vec_{S}(Z):=[{\bm z}^{\T}_1,{\bm z}^{\T}_2,\ldots, {\bm z}^{\T}_m]^{\T}\in \R^{\frac{m(m+1)}{2}},$$ 
    where 
    ${\bm z}_1=[z_{11},z_{21},\ldots, z_{m1}]^{\T}\in \R^m,$ ${\bm z}_2=[z_{22},z_{32},\ldots, z_{m2}]^{\T}\in \R^{m-1},  \ldots,  {\bm z}_{m-1}=[z_{(m-1)(m-1)},$ $ z_{m(m-1)}]^{\T}\in \R^2,$ $ {\bm z}_m=[z_{mm}]\in\R.$
\end{definition}
\begin{definition}
    Let  $Z\in \mathbb{SKR}^{m\times m},$ then we define its generator vector by $$\vec_{SK}(Z):=[{\bm z}^{\T}_1,{\bm z}^{\T}_2,\ldots, {\bm z}^{\T}_{m-1}]^{\T}\in \R^{\frac{m(m-1)}{2}},$$ 
    where ${\bm z}_1=[z_{21},\ldots, z_{m1}]^{\T}\in \R^{m-1},$ ${\bm z}_2=[z_{32},\ldots, z_{m2}]^{\T}\in \R^{m-2},$ $ \ldots,$ $ {\bm z}_{m-1}=[z_{(m-1)(m-1)}]^{\T}\in \R.$
\end{definition}
 The following two lemmas can be also found in \cite{Herm2016}.
\begin{lemma}\label{lemma1}
    Let $M\in \mathbb{SR}^{m\times m}.$ Then $\vec(M) = \mathcal{J}^m_{S}\vec_{S}(M),$
    where $$\mathcal{J}^m_{S}=\bmatrix{\mathcal{J}^{(1)}_{S}& \mathcal{J}^{(2)}_{S} & \cdots &\mathcal{J}^{(m)}_{S}}\in \R^{m^2\times \frac{m(m+1)}{2}},$$  $\mathcal{J}^{(i)}_S\in \R^{m\times (m-i+1)}$  are defined by
  {\footnotesize  \begin{align*}
\mathcal{J}^{(1)}_{S}=\bmatrix{e_1^{m}& e_2^{m}&e_3^{m} & \cdots&e_{m-1}^{m}& e_m^{m}\\
    0 &e_1^{m} &0& \cdots &\cdots&0\\
    0&0&e_1^{m}&\cdots&\cdots&0\\
    \vdots& \vdots &\vdots &\ddots& &\vdots\\ 0&0&0 &\cdots&  e_1^{m}&0\\
    0&0&0 &\cdots& 0& e_1^{m}
    },\, 
    \mathcal{J}^{(2)}_{S}=\bmatrix{0 & 0&\cdots&\cdots &0\\
    e_2^{m}&e_3^{m}& \cdots & \cdots& e_m^{m}\\
    0 &e_2^{m} &0& \cdots &0\\
    0&0&e_2^{m}&\cdots&0\\
    \vdots& \vdots &\ddots &\ddots& \vdots\\
    0&\cdots &\cdots& 0& e_2^{m}
    },\,
   \ldots,~ \mathcal{J}^{(m)}_{S}=\bmatrix{0\\ 0\\ 0\\\vdots\\ 0\\e_m^{m}},
    \end{align*}
    }
    and $e^m_i$ denotes the $i^{th}$ column of the identity matrix $I_m.$
\end{lemma}
\begin{lemma}\label{lemma2}
    Let $M\in \mathbb{SKR}^{m\times m}.$ Then $\vec(M) = \mathcal{J}^m_{SK}\vec_{SK}(M),$
    where $$\mathcal{J}^m_{SK}=\bmatrix{\mathcal{J}^{(1)}_{SK}& \mathcal{J}^{(2)}_{SK} & \cdots \mathcal{J}^{(m-1)}_{SK}}\in \R^{m^2\times \frac{m(m-1)}{2}},$$  $\mathcal{J}^{(i)}_{SK}\in \R^{m\times (m-i)}$  are defined by  
     {\footnotesize  \begin{align*}
\mathcal{J}^{(1)}_{SK}=\bmatrix{e_2^{m}& e_3^{m} & \cdots&e_{m-1}^{m}& e_m^{m}\\
    -e_1^{m} &0& \cdots &\cdots&0\\
    0&-e_1^{m}&\cdots&\cdots&0\\
     \vdots &\vdots &\ddots& &\vdots\\
     0&0 &\cdots&  -e_1^{m}&0\\
    0&0 &\cdots& 0& -e_1^{m}
    }, 
    \mathcal{J}^{(2)}_{SK}=\bmatrix{ 0&\cdots&\cdots &0\\
    e_3^{m}& \cdots & \cdots& e_m^{m}\\
   -e_2^{m} &0& \cdots &0\\
    0&-e_2^{m}&\cdots&0\\
     \vdots &\ddots &\ddots& \vdots\\
    \cdots &\cdots& 0& -e_2^{m}
    },
   \ldots,~ \mathcal{J}^{(m-1)}_{SK}=\bmatrix{0\\ 0\\ 0\\\vdots\\ e_m^{m}\\-e_{m-1}^{m}},
    \end{align*}
    }
   and $e^m_i$ denotes the $i^{th}$ column of the identity matrix $I_m.$ 
\end{lemma}
Next, we present an important lemma that is crucial for computing structured BEs while preserving the sparsity pattern.
\begin{lemma}\label{lemma2:SEC2}
Assume $M\in \R^{m\times m}$ and $X\in \mathbb{HC}^{n\times n}.$ Then the following holds:
    \begin{enumerate}
        \item When $M\in \mathbb{SR}^{m\times m}.$ Then, we have
        \begin{equation}\label{eq24}
        \vec(M\odot \Theta_{X})=  \mathcal{J}^m_{S}\Phi_X\vec_{S}(M\odot \Theta_{X}),
        \end{equation}
        where $\Phi_X=\diag(\vec_{S}(\Theta_X)).$
       \item When $M\in \mathbb{SKR}^{m\times m}.$ Then, we have  
       \begin{equation}\label{eq25} 
       \vec(M\odot \Theta_{X})=  \mathcal{J}^m_{SK}\Psi_X\vec_{SK}(M \odot \Theta_{X}),
       \end{equation}
    \end{enumerate}
    
    where 
$\Psi_X=\diag([\Theta_X(1,2:m), \Theta_X(2, 3:m),\ldots, \Theta_X( m-1:m)]^{\T}).$
\end{lemma}
\begin{proof}
   Let $M\in \mathbb{SR}^{m\times m}$ and $X\in \mathbb{HC}^{n\times n}.$ By definition of the matrix $\Theta_{X}$, we have $\Theta_{X}\in \mathbb{SR}^{m\times m}$ and consequently, $M\odot \Theta_{X}\in \mathbb{SR}^{m\times m}$. Then 
\begin{align*}
    \vec(M\odot \Theta_{X})&=\mathcal{J}^m_{S}\vec_{S}(M\odot \Theta_{X})\\
    &=\mathcal{J}^m_{S}\Phi_{X}\vec_{S}(M\odot \Theta_{X}).
\end{align*}
Hence, \eqref{eq24} follows.   Similarly, we can prove \eqref{eq25} when $M\in \mathbb{SKR}^{m\times m}.$
\end{proof} 
\begin{remark}
    By consider $X={\bf 1}_{m\times m},$ Lemma \ref{lemma2:SEC2} reduces to those given in Lemmas \ref{lemma1} and \ref{lemma2}, where ${\bf 1}_{m\times m}$ denotes $m\times m$ matrix with all entries equal to $1$. 
\end{remark}
\begin{remark}
    When $M, X\in \C^{m\times n},$ we have  $\vec(M\odot \Theta_X)=\Sigma_X\vec(M\odot \Theta_X ),$ where
    $\Sigma_X=\diag(\vec(\Theta_X)).$
\end{remark}
To illustrate Lemma \ref{lemma2:SEC2}, we consider the following example.
\begin{exam}
    Let $M\in \mathbb{SR}^{3\times 3}$ and $X\in \mathbb{HC}^{3\times 3}$ be given by
    \begin{equation*}
        M=\bmatrix{1 & 2& 3\\ 2& 4 &5 \\ 3& 5 & 6}~\text{and}~  X=\bmatrix{7 & 8+i& 0\\ 8-i& 9 &10+2i \\ 0 & 10-2i & 0}.
    \end{equation*}
    Then,
        $M\odot \Theta_X=\bmatrix{1 & 2& 0\\ 2& 4 &5 \\ 0& 5 & 0}$ and we have
    \begin{align}\vec(M\odot \Theta_X)&=\bmatrix{e_1^3& e_2^3& e_3^3& 0& 0&0\\ 0&e_1^3& 0&e_2^3&e_3^3& 0\\ 0&0 &e_1^3&0 &e_2^3&e_3^3}\bmatrix{1\\2\\ 0\\4\\5\\ 0}=\mathcal{J}^3_{S}\vec_{S}(M\odot \Theta_X)\\
    &=\mathcal{J}^3_{S}\bmatrix{1 &0& 0&0 &0&0\\ 0&1&0& 0&0 &0\\ 0&0&0& 0&0 &0\\ 0&0&0& 1&0 &0\\ 0&0&0& 0&1 &0\\ 0&0&0& 0&0 &0} \bmatrix{1\\2\\ 0\\4\\5\\ 0}\\
    &=\mathcal{J}_{S}^3\Phi_X\vec_{S}(M\odot \Theta_X).
    \end{align}
\end{exam}

\section{Computation of Structured BEs}\label{sec:SBE}
In this section, we derive the closed-form expressions for  $\bm{\xi}^{\mathcal{G}_i}_{\tt sps}(\widehat{\bm{u}},\widehat{\bm{p}})$ for $i=1,2,3,$ by preserving the sparsity of the coefficient matrix and the Hermitian structure of 
 the matrix $E$ and $G,$ respectively.  Moreover, we provide the optimal perturbations for which the structured BE is attained. The following lemma plays a crucial role in the computation of the structured BE.
\begin{lemma}\label{lemm1:sec3}
 Consider the following  $2\times 2$ block linear system:
  \begin{align}\label{eq:31}
      \bmatrix{E+\D E & (F+\D F)^{*}\\ H+\D H & G+\D G}\bmatrix{\bm{u}\\ \bm{p}}=\bmatrix{q+\D q\\ r+\D r}.
  \end{align} 
  Then\eqref{eq:31} can  be reformulated as the following systems of linear equations:
  \begin{align}
  \left\{\begin{array}{c}
  \mathfrak{R}(\D E)\RR(\u)-\I(\D E)\mathfrak{I}(\u)+ \mathfrak{R}(\D F)^{\T}\RR(\p)+ \I(\D F)^{\T}\mathfrak{I}(\p) - \RR(\D q)=\RR(Q),\\
   \RR(\D E)\I(\u)+\I(\D E)\RR(\u)+\RR(\D F)^{\T}\I(\p)-\I(\D F)^{\T}\RR(\p)-\I(\D q)=\I(Q),
  \end{array}
  \right.
  \end{align}
  and 
  \begin{align}
  \left\{\begin{array}{c}
  \mathfrak{R}(\D H)\RR(\u)-\I(\D H)\mathfrak{I}(\u)+ \mathfrak{R}(\D G)\RR(\p)- \I(\D G)\mathfrak{I}(\p) - \RR(\D r)=\RR(R),\\
   \RR(\D H)\I(\u)+\I(\D H)\RR(\u)+\RR(\D G)\I(\p)+\I(\D G)\RR(\p)-\I(\D r)=\I(R),
  \end{array}
  \right.
  \end{align}
   where \begin{align}
   Q=q-E\u-F^*\p~~\text{and}~~ R=r-H\u-G\p.
  \end{align}
\end{lemma}
\begin{proof}
   The \(2 \times 2\) block linear system \eqref{eq:31} can be equivalently expressed as:
    %
    \begin{align}\label{eq1:32}
        &\D E \u+\D F^{*} \p-\D q=q -E\u -F^*\p,\\ \label{eq1:33}
        &\D H\u+\D G \p-\D r = r-H \u-G\p.
    \end{align}
    The proof follows by separating the real and imaginary parts from \eqref{eq1:32} and \eqref{eq1:33}.
\end{proof}
\subsection{Computation of the structured BE for case (\textit{i})}
In this subsection, we derive an explicit expression for the structured BE $\bm{\xi}^{\mathcal{G}_1}_{\tt sps} (\widehat{\bm{u}},\widehat{\bm{p}})$ by preserving the structure of $E\in \mathbb{HC}^{n\times n}$, ${F\in \C^{m\times n}}$, and $G\in \mathbb{C}^{m\times m}$ and maintaining the sparsity pattern of the coefficient matrix $\mathfrak{B}$ within the perturbation matrices.

Prior to stating the main theorem of this subsection, we  construct the following matrices:
Let $\mathfrak{D}_{S,n}\in \R^{\frac{n(n+1)}{2}\times \frac{n(n+1)}{2}}$ and $\mathfrak{D}_{SK,n}\in \R^{\frac{n(n-1)}{2}\times \frac{n(n-1)}{2}} $ be the diagonal matrices with
 $$\mathfrak{D}_{S,n}(j,j):=\left\{ \begin{array}{lcl}
		1, & \mbox{for} & j=\frac{(2n-(i-2))(i-1)}{2}+1,\,\, i=1,2,\ldots,n, \\ \sqrt{2},&  &\text{otherwise},
	\end{array}\right.$$
and $$\mathfrak{D}_{SK,n}(j,j):=\sqrt{2},~~ \text{for}~ j=1,2,\ldots, \frac{n(n-1)}{2}.$$
 Further, set
\begin{equation}
N_1:=\mathcal{J}^n_{S}\Phi_{E}\mathfrak{D}^{-1}_{S,n}\in \R^{n^2\times \frac{n(n+1)}{2}} ~~\text{and}~~N_2:=\mathcal{J}^n_{SK}\Psi_{E}\mathfrak{D}^{-1}_{SK,n}\in \R^{n^2\times \frac{n(n-1)}{2}}.
\end{equation}
Let $\bm{s}=n^2+2mn+2m^2$,  $\mathbf{X}_1\in \R^{2n\times \bm{s}}$ and $\mathbf{X}_2\in \R^{2m\times \bm{s}}$ be defined as follows:
 $$\mathbf{X}_1:= [\widehat{\mathbf{X}}_1~~ {0}_{2n\times 2m^2}]~~\text{and}~~ \mathbf{X}_2:= [ {0}_{2m\times n^2}~~ \widehat{\mathbf{X}}_2],$$
 where
{\footnotesize\begin{align*}
    \widehat{\mathbf{X}}_1= \bmatrix{\alpha_1^{-1}(\mathfrak{R}(\widehat{\u})^{\T}\otimes I_n)N_1 &-\alpha_1^{-1}(\mathfrak{I}(\widehat{\u})^{\T}\otimes I_n)N_2& \alpha_2^{-1} (I_n\otimes\mathfrak{R}(\widehat{\p})^{\T})\Sigma_F & \alpha_2^{-1}(I_n\otimes\mathfrak{I}(\widehat{\p})^{\T})\Sigma_F\\
     \alpha_1^{-1} (\mathfrak{I}(\widehat{\u})^{\T}\otimes I_n)N_1 &  \alpha_1^{-1}(\mathfrak{R}(\widehat{\u})^{\T}\otimes I_n)N_2&  \alpha_2^{-1}(I_n\otimes\mathfrak{I}(\widehat{\p})^{\T})\Sigma_F& -  \alpha_2^{-1}(I_n\otimes\mathfrak{R}(\widehat{\p})^{\T})\Sigma_F } 
  \end{align*}}
  and
  {\footnotesize\begin{align*}
    \widehat{\mathbf{X}}_2= \bmatrix{ \alpha_2^{-1}(\mathfrak{R}(\widehat{\u})^{\T}\otimes I_m)\Sigma_F &-\alpha_2^{-1}(\mathfrak{I}(\widehat{\u})^{\T}\otimes I_m)\Sigma_F& \alpha_3^{-1}(\mathfrak{R}(\widehat{\p})^{\T}\otimes I_m)\Sigma_G & -\alpha_3^{-1}(\mathfrak{I}(\widehat{\p})^{\T}\otimes I_m)\Sigma_G\\
      \alpha_2^{-1}(\mathfrak{I}(\widehat{\u})^{\T}\otimes I_m)\Sigma_F &\alpha_2^{-1} (\mathfrak{R}(\widehat{\u})^{\T}\otimes I_m)\Sigma_F&  \alpha_3^{-1}(\mathfrak{I}(\widehat{\p})^{\T} \otimes I_m)\Sigma_G& \alpha_3^{-1} (\mathfrak{R}(\widehat{\p})^{\T}\otimes I_m)\Sigma_G}.
  \end{align*}
  }
%
Set \begin{align}\label{eq:312}
   \D \mathcal{Y}:= \bmatrix{\alpha_1\mathfrak{D}_{S, n}\vec_{S}(\mathfrak{R}(\D E\odot
  \Theta_E))\\ \alpha_1\mathfrak{D}_{SK,n}\vec_{SK}(\mathfrak{I}(\D E\odot
  \Theta_E))\\ \alpha_2\vec(\mathfrak{R}(\D F\odot
  \Theta_F))\\ \alpha_2\vec(\mathfrak{I}(\D F\odot
  \Theta_F))\\ \alpha_3\vec(\mathfrak{R}(\D G\odot
  \Theta_G))\\ \alpha_3\vec(\mathfrak{I}(\D G\odot
  \Theta_G))}\in \R^{\bm s}~~ \text{and} ~~ \D \mathcal{Z}:=\bmatrix{\beta_1\RR(\D q)\\ \beta_1\I(\D q)\\\beta_2 \RR(\D r)\\\beta_2 \I(\D r)}\in \R^{2(n+m)}.
\end{align}
 Note that 
\begin{align}
 \nonumber    \alpha_1^2\|E\|^2_{\mathsf{F}}&= \alpha_1^2\|\RR(E)\|_{\mathsf{F}}^2+\alpha_1^2\|\I(E)\|_{\mathsf{F}}^2\\
     &=\|\alpha_1\mathfrak{D}_{S,n}\vec_{S}(\mathfrak{R}\D E)\|_2^2+\|\alpha_1\mathfrak{D}_{SK,n}\vec_{SK}(\mathfrak{I}\D E)\|_2^2\\ \nonumber     &=\left\|\bmatrix{\alpha_1\mathfrak{D}_{S,n}\vec_{S}(\mathfrak{R}\D E)\\ \alpha_1\mathfrak{D}_{SK,n}\vec_{SK}(\mathfrak{I}\D E)}\right\|_2^2.
 \end{align}

The following theorem provides a compact representation of the structured BE $ \bm{\xi}_{\tt sps}^{\mathcal{G}_1}(\widehat{\bm{u}},\widehat{\bm{p}})$.
\begin{theorem}\label{th1:case1}
  Assume that $E\in \mathbb{HC}^{n\times n},$ $F\in \C^{m\times n},$ $G\in \C^{m\times m}$  and $\widehat{\bm{x}}=[\widehat{\bm{u}}^{\T},\widehat{\bm{p}}^{\T}]^{\T}$ is a computed solution of the GSPP \eqref{eq11}. Then, the structured BE $\bm{\xi}^{\mathcal{G}_1}_{\tt sps}(\widehat{\bm{u}},\widehat{\bm{p}})$ with preserving sparsity is given by  
    \begin{eqnarray}
      \bm{\xi}_{\tt sps}^{\mathcal{G}_1}(\widehat{\bm{u}},\widehat{\bm{p}})= \left\| \bmatrix{\mathbf{X}_1& \mathcal{I}_1\\ \mathbf{X}_2& \mathcal{I}_2}^{\T}\left(\bmatrix{\mathbf{X}_1& \mathcal{I}_1\\ \mathbf{X}_2& \mathcal{I}_2}\bmatrix{\mathbf{X}_1& \mathcal{I}_1\\ \mathbf{X}_2& \mathcal{I}_2}^{\T}\right)^{-1}\bmatrix{\RR(Q)\\ \I(Q)\\ \RR(R)\\ \I(R)}\right\|_2,
    \end{eqnarray}
    where $Q=q-E\u-F^*\p$, $R=r-F\u-G\p,$ $\mathcal{I}_1=\bmatrix{ -{\beta_1}^{-1} I_{2n} & 0_{2n\times 2m}}\in \R^{2n\times 2(n+m)}$ and $\mathcal{I}_2=\bmatrix{ 0_{2m \times 2n}&-{\beta_2}^{-1} I_{2m}}\in \R^{2m\times 2(n+m)}.$
\end{theorem}
\begin{proof}
    Let $\widehat{\bm{x}}=[\widehat{\bm{u}}^{\T},\widehat{\bm{p}}^{\T}]^{\T}$ be a computed solution of the GSPP \eqref{eq11} with $E\in \mathbb{HC}^{n\times n}$. Then, we need to find the perturbations $\D q\in \C^{n},$  $\D r\in \C^{m}$ and sparsity preserving perturbation matrices $\D E\in \mathbb{HC}^{n\times n},$ $\D F\in \C^{m\times n}$
 and $\D G\in \C^{m\times m}$ so that \eqref{s3:eq24} holds.    Therefore, we replace $\D E,$ $\D F$ and $\D G$ with 
    $\D E \odot \Theta_E,$ $\D F\odot \Theta_F$ and  $\D G\odot \Theta_G,$ respectively, such the following holds:
    \begin{align}
       \bmatrix{E+(\D E \odot \Theta_E) & (F+(\D F \odot \Theta_F))^{*}\\ F+(\D F \odot \Theta_F) & G+(\D G \odot \Theta_G)}\bmatrix{\widehat{\bm{u}}\\ \widehat{\bm{p}}}=\bmatrix{q+\D q\\ r+\D r}.
  \end{align} 
  Then using Lemma \ref{lemm1:sec3} with $H=F$, we have
   \begin{align}\label{eq1:th1}
  \left\{\begin{array}{r}
  \mathfrak{R}(\D E\odot \Theta_E)\RR({\widehat{\u}})-\I(\D E\odot \Theta_E)\mathfrak{I}({\widehat{\u}})+ \mathfrak{R}(\D F\odot \Theta_F)^{\T}\RR({\widehat{\p}})\\
  + \I(\D F\odot \Theta_F)^{\T}\mathfrak{I}({\widehat{\p}}) - \RR(\D q)=\RR(Q),\\
   \RR(\D E\odot \Theta_E)\I({\widehat{\u}})+\I(\D E\odot \Theta_E)\RR({\widehat{\u}})+\RR(\D F\odot \Theta_F)^{\T}\I({\widehat{\p}})\\
   -\I(\D F\odot \Theta_F)^{\T}\RR({\widehat{\p}})-\I(\D q)=\I(Q),
  \end{array}
  \right.
  \end{align}
  and 
  \begin{align}\label{eq11:th1}
  \left\{\begin{array}{r}
  \mathfrak{R}(\D F\odot \Theta_F)\RR({\widehat{\u}})-\I(\D F\odot \Theta_F)\mathfrak{I}({\widehat{\u}})+ \mathfrak{R}(\D G\odot \Theta_G)\RR({\widehat{\p}})\\
  - \I(\D G\odot \Theta_G)\mathfrak{I}({\widehat{\p}}) - \RR(\D r)=\RR(R),\\
   \RR(\D F\odot \Theta_F)\I({\widehat{\u}})+\I(\D F\odot \Theta_F)\RR({\widehat{\u}})+\RR(\D G\odot \Theta_G)\I({\widehat{\p}})\\
   +\I(\D G\odot \Theta_G)\RR({\widehat{\p}})-\I(\D r)=\I(R).
  \end{array}
  \right.
  \end{align}
 Now, using the properties of the vec operator and Kronecker product on \eqref{eq1:th1}, we obtain
  
   \begin{align}\label{eq2:th1}
  \left\{\begin{array}{r}
 (\mathfrak{R}(\widehat{\u})^{\T}\otimes I_n)\vec(\mathfrak{R}(\D E\odot
  \Theta_E))-(\mathfrak{I}(\widehat{\u})^{\T}\otimes I_n)\vec(\mathfrak{I}(\D E\odot
  \Theta_E))\\
  + ( I_n\otimes\mathfrak{R}(\widehat{\p})^{\T})\vec(\mathfrak{R}(\D F\odot
  \Theta_F))
  {\color{blue}+} ( I_n\otimes\mathfrak{I}(\widehat{\p})^{\T})\vec(\mathfrak{I}(\D F\odot
  \Theta_F)) - \RR(\D q)=\RR(Q),\\
   (\mathfrak{I}(\widehat{\u})^{\T}\otimes I_n)\vec(\mathfrak{R}(\D E\odot
  \Theta_E))+(\mathfrak{R}(\widehat{\u})^{\T}\otimes I_n)\vec(\mathfrak{I}(\D E\odot
  \Theta_E)\\
  + ( I_n\otimes\mathfrak{I}(\widehat{\p})^{\T})\vec(\mathfrak{R}(\D F\odot
  \Theta_F))
  - ( I_n\otimes\mathfrak{R}(\widehat{\p})^{\T})\vec(\mathfrak{I}(\D F\odot
  \Theta_F)) - \I(\D q)=\I(Q).
  \end{array}
  \right.
  \end{align}
 As $\D E\in \mathbb{HC}^{n\times n},$ we have $\RR(\D E)\in \mathbb{SR}^{n\times n}$ and $\I(\D E)\in \mathbb{SKR}^{n\times n}.$  Further, we have $\D E \odot \Theta_E\in \mathbb{HC}^{n\times n},$ $\RR(\D E\odot \Theta_E)\in \mathbb{SR}^{n\times n}$ and $\I(\D E\odot \Theta_E)\in \mathbb{SKR}^{n\times n}.$ Hence, using Lemma \ref{lemma2:SEC2} on \eqref{eq2:th1}, we get
    \begin{align}\label{eq:321}
  \left\{\begin{array}{r}
 (\mathfrak{R}(\widehat{\u})^{\T}\otimes I_n)\mathcal{J}^n_{S}\Phi_{E}\vec_{S}(\mathfrak{R}(\D E\odot
  \Theta_E))-(\mathfrak{I}(\widehat{\u})^{\T}\otimes I_n)\mathcal{J}^n_{SK}\Psi_{E}\vec_{SK}(\mathfrak{I}(\D E\odot
  \Theta_E))\\
  + ( I_n\otimes\mathfrak{R}(\widehat{\p})^{\T})\Sigma_F\vec(\mathfrak{R}(\D F\odot
  \Theta_F))
  {\color{blue}+} ( I_n\otimes\mathfrak{I}(\widehat{\p})^{\T})\Sigma_F\vec(\mathfrak{I}(\D F\odot
  \Theta_F)) - \RR(\D q)=\RR(Q),\\
   (\mathfrak{I}(\widehat{\u})^{\T}\otimes I_n)\mathcal{J}^n_{S}\Phi_{E}\vec_{S}(\mathfrak{R}(\D E\odot
  \Theta_E))+(\mathfrak{R}(\widehat{\u})^{\T}\otimes I_n)\mathcal{J}^n _{SK}\Psi_{E}\vec_{SK}(\mathfrak{I}(\D E\odot
  \Theta_E)\\
  + ( I_n\otimes\mathfrak{I}(\widehat{\p})^{\T})\Sigma_F\vec(\mathfrak{R}(\D F\odot
  \Theta_F))
  - ( I_n\otimes\mathfrak{R}(\widehat{\p})^{\T})\Sigma_F\vec(\mathfrak{I}(\D F\odot
  \Theta_F)) - \I(\D q)=\I(Q).
  \end{array}
  \right.
  \end{align}
  %
%
 Then, \eqref{eq:321} can be reformulate as
\begin{align}\label{eq322}
    \mathbf{X}_1\D \mathcal{Y}+ \bmatrix{ -{\color{blue}\beta_1}^{-1} I_{2n} & 0_{2m}} \D \mathcal{Z}=\bmatrix{\RR(Q)\\ \I(Q)}.
\end{align}
%
  %
In a similar manner, from \eqref{eq11:th1}, we can deduce the following:
  \begin{align}\label{eq323}
      \mathbf{X}_2 \D \mathcal{Y} + \bmatrix{ 0_{2n}&-{\color{blue}\beta_2}^{-1} I_{2m}} \D \mathcal{Z}=\bmatrix{ \RR(R)\\ \I(R)}.
  \end{align}
  Therefore, combining \eqref{eq322} and \eqref{eq323}, we get
  \begin{align}\label{eq319}
   \bmatrix{\mathbf{X}_1\\ \mathbf{X}_2}\D \mathcal{Y}+\bmatrix{\mathcal{I}_1\\ \mathcal{I}_2} \D \mathcal{Z}=\bmatrix{\RR(Q)\\ \I(Q)\\ \RR(R)\\ \I(R)}
     \Longleftrightarrow
      \bmatrix{\mathbf{X}_1& \mathcal{I}_1\\ \mathbf{X}_2& \mathcal{I}_2}\bmatrix{\D \mathcal{Y}\\ \D \mathcal{Z}}=\bmatrix{\RR(Q)\\ \I(Q)\\ \RR(R)\\ \I(R)}.
  \end{align}
 Since the matrix \( \bmatrix{\mathbf{X}_1 & \mathcal{I}_1 \\ \mathbf{X}_2 & \mathcal{I}_2} \) has full row rank, the consistency condition in Lemma \ref{sec2:lemma} is satisfied, and the minimal norm solution of \eqref{eq319} is expressed as:
  \begin{align}\label{eq321}
    \nonumber  \bmatrix{\D \mathcal{Y}\\ \D \mathcal{Z}}_{\tt opt}&=\bmatrix{\mathbf{X}_1& \mathcal{I}_1\\ \mathbf{X}_2& \mathcal{I}_2}^{\dagger}\bmatrix{\RR(Q)\\ \I(Q)\\ \RR(R)\\ \I(R)}\\
    &=\bmatrix{\mathbf{X}_1& \mathcal{I}_1\\ \mathbf{X}_2& \mathcal{I}_2}^{\T}\left(\bmatrix{\mathbf{X}_1& \mathcal{I}_1\\ \mathbf{X}_2& \mathcal{I}_2}\bmatrix{\mathbf{X}_1& \mathcal{I}_1\\ \mathbf{X}_2& \mathcal{I}_2}^{\T}\right)^{-1}\bmatrix{\RR(Q)\\ \I(Q)\\ \RR(R)\\ \I(R)}. 
      \end{align}
  According to the Definition \ref{def:SBE}, we have
  \begin{align}
\nonumber\bm{\xi}^{\mathcal{G}_1}_{\tt sps}(\widehat{\bm{u}},\widehat{\bm{p}})&=\displaystyle{\min_{\left(\begin{array}{c}
       \D E\odot  E,\D F\odot  F,\\
     \D G\odot  G, \D q, \D r 
     \end{array}\right)\in\, \mathcal{G}_1}} {{\bm\zeta}^{{\sigma}_1}(\D E\odot  E,\D F\odot  F,\D G\odot  G, \D q, \D r)}\\ \nonumber
&=\min_{\left(\begin{array}{c}
       \D E\odot  E,\D F\odot  F,\\
     \D G\odot
  \Theta_G, \D q, \D r 
     \end{array}\right)\in\, \mathcal{G}_1} \left\|\bmatrix{\alpha_1\vec(\mathfrak{R}(\D E\odot
  \Theta_E))\\ \alpha_1\vec(\mathfrak{I}(\D E\odot
  \Theta_E))\\ \alpha_2\vec(\mathfrak{R}(\D F\odot
  \Theta_F))\\ \alpha_2\vec(\mathfrak{I}(\D F\odot
  \Theta_F))\\ \alpha_3\vec(\mathfrak{R}(\D G\odot
  \Theta_G))\\ \alpha_3\vec(\mathfrak{I}(\D G\odot
  \Theta_G))\\ \beta_1\RR(\D q)\\ \beta_1\I(\D q)\\\beta_2 \RR(\D r)\\\beta_2 \I(\D r)}\right\|_2\\ \nonumber
  &=\min_{\left(\begin{array}{c}
       \D E\odot  E,\D F\odot  F,\\ 
     \D G\odot
  \Theta_G, \D q, \D r 
     \end{array}\right)\in\, \mathcal{G}_1} \left\|\bmatrix{\alpha_1\mathfrak{D}_{S,n}\vec_{S}(\mathfrak{R}(\D E\odot
  \Theta_E))\\ \alpha_1\mathfrak{D}_{SK,n}\vec_{SK}(\mathfrak{I}(\D E\odot
  \Theta_E))\\ \alpha_2\vec(\mathfrak{R}(\D F\odot
  \Theta_F))\\ \alpha_2\vec(\mathfrak{I}(\D F\odot
  \Theta_F))\\ \alpha_3\vec(\mathfrak{R}(\D G\odot
  \Theta_G))\\ \alpha_3\vec(\mathfrak{I}(\D G\odot
  \Theta_G))\\ \beta_1\RR(\D q)\\ \beta_1\I(\D q)\\\beta_2 \RR(\D r)\\\beta_2 \I(\D r)}\right\|_2\\ \nonumber
     &=\min\left\{ \left\| \bmatrix{\D \mathcal{Y}\\ \D \mathcal{Z}}
 \right\|_2 {\Bigg |} \bmatrix{\mathbf{X}_1& \mathcal{I}_1\\ \mathbf{X}_2& \mathcal{I}_2}\bmatrix{\D \mathcal{Y}\\ \D \mathcal{Z}}=\bmatrix{\RR(Q)\\ \I(Q)\\ \RR(R)\\ \I(R)} \right\}\\ 
 &= \left\| \bmatrix{\D \mathcal{Y}\\ \D \mathcal{Z}}_{\tt{opt}}
 \right\|_2.
     \end{align}
Consequently, the proof is conclusive.
\end{proof}

\begin{remark}\label{remark1}
     Using \eqref{eq:312} and \eqref{eq321}, 
 the optimal perturbation matrix  $\D E^{\tt sps}_{\tt{opt}}$ is obtained as follows:
   $$\D E^{\tt sps}_{\tt{opt}}=\RR(\D E^{\tt sps}_{\tt{opt}})+i\I(\D E^{\tt sps}_{\tt{opt}}),$$
   where 
   $$\vec_{S}(\RR(\D E^{\tt sps}_{\tt{opt}}))=\alpha_1^{-1}\mathfrak{D}_{S,n}^{-1}\bmatrix{I_{\bm{n}_1} & 0_{\bm{n}_1\times (\bm{l}-\bm{n}_1)}}\bmatrix{\D \mathcal{Y}\\ \D \mathcal{Z}}_{\tt{opt }},$$  $$\vec_{SK}(\I(\D E^{\tt sps}_{\tt{opt}}))=\alpha_1^{-1}\mathfrak{D}_{SK,n}^{-1}\bmatrix{0_{\bm{n}_2\times \bm{n}_1}&I_{\bm{n}_2} & 0_{\bm{n}_2\times (\bm{l}-n^2)}}\bmatrix{\D \mathcal{Y}\\ \D \mathcal{Z}}_{\tt{opt }},$$
   $\bm{n}_1=\frac{n(n+1)}{2},$  $\bm{n}_2=\frac{n(n-1)}{2},$ and $\bm{l}=\bm{s}+2(m+n).$
   Similarly, optimal backward perturbation matrix  $\D F^{\tt sps}_{\tt{opt}}$, $\D G^{\tt sps}_{\tt{opt}},$ $\D r_{\tt{opt}},$ and $\D q_{\tt{opt}}$ are given by
   $$\vec(\D F^{\tt sps}_{\tt{opt}})=\alpha_2^{-1}\bmatrix{0_{nm\times n^2} &I_{nm} & i I_{nm}& 0_{nm\times  (\bm{l}-n^2-2nm)}}\bmatrix{\D \mathcal{Y}\\ \D \mathcal{Z}}_{\tt{opt }},$$ 
   $$\vec(\D G^{\tt sps}_{\tt{opt}})=\alpha_3^{-1}\bmatrix{0_{m^2\times (n^2+2nm)} &I_{m^2} &i I_{m^2}& 0_{m^2\times 2(m+n)}}\bmatrix{\D \mathcal{Y}\\ \D \mathcal{Z}}_{\tt{opt}},$$  
   $$\D q_{\tt{opt}}=\beta_1^{-1}\bmatrix{0_{n\times ( \bm{l}-2(n+m))}&I_n &iI_n& 0_{n\times 2m}}\bmatrix{\D \mathcal{Y}\\ \D \mathcal{Z}}_{\tt{opt }},$$
and $$\D r_{\tt{opt}}=\beta_2^{-1}\bmatrix{0_{m\times (\bm{l}-2m)}&I_m &iI_m}\bmatrix{\D \mathcal{Y}\\ \D \mathcal{Z}}_{\tt{opt }}.$$

\end{remark}

We now compute a compact formula for the structured BE for case (\textit{i}), in which the sparsity pattern of  \( \mathfrak{B} \) is not retained.
\begin{corollary}\label{coro1}
 Assume that $E\in \mathbb{HC}^{n\times n},$ $F\in \C^{m\times n}$ and $G\in \C^{m\times m}$  and $\widehat{\bm{x}}=[\widehat{\bm{u}}^{\T},\widehat{\bm{p}}^{\T}]^{\T}$ is a computed solution of the GSPP \eqref{eq11}. Then, the structured BE $\bm{\xi}^{\mathcal{G}_1}(\widehat{\bm{u}},\widehat{\bm{p}})$ without preserving sparsity is expressed as  
    \begin{eqnarray*}
      \bm{\xi}^{\mathcal{G}_1}(\widehat{\bm{u}},\widehat{\bm{p}})= \left\| \bmatrix{\mathbf{Y}_1& \mathcal{I}_1\\ \mathbf{Y}_2& \mathcal{I}_2}^{\T}\left(\bmatrix{\mathbf{Y}_1& \mathcal{I}_1\\ \mathbf{Y}_2& \mathcal{I}_2}\bmatrix{\mathbf{Y}_1& \mathcal{I}_1\\ \mathbf{Y}_2& \mathcal{I}_2}^{\T}\right)^{-1}\bmatrix{\RR(Q)\\ \I(Q)\\ \RR(R)\\ \I(R)}\right\|_2,
    \end{eqnarray*}
    where $$\mathbf{Y}_1= [\widehat{\mathbf{Y}}_1~~ {0}_{2n\times 2m^2}],~~ \mathbf{Y}_2= [ {0}_{2m^2\times n^2}~~ \widehat{\mathbf{Y}}_2],$$
    {\footnotesize\begin{align*}
    \widehat{\mathbf{Y}}_1= \bmatrix{\alpha_1^{-1}(\mathfrak{R}(\widehat{\u})^{\T}\otimes I_n)\mathcal{J}^n_{S}\mathfrak{D}^{-1}_{S,n} &-\alpha_1^{-1}(\mathfrak{I}(\widehat{\u})^{\T}\otimes I_n)\mathcal{J}^n_{SK}\mathfrak{D}^{-1}_{SK,n}& \alpha_2^{-1} (I_n\otimes\mathfrak{R}(\widehat{\p})^{\T}) & \alpha_2^{-1}(I_n\otimes\mathfrak{I}(\widehat{\p})^{\T})\\
     \alpha_1^{-1} (\mathfrak{I}(\widehat{\u})^{\T}\otimes I_n)\mathcal{J}^n_{S}\mathfrak{D}^{-1}_{S,n} &  \alpha_1^{-1}(\mathfrak{R}(\widehat{\u})^{\T}\otimes I_n)\mathcal{J}^n_{SK}\mathfrak{D}^{-1}_{SK,n}&  \alpha_2^{-1}(I_n\otimes\mathfrak{I}(\widehat{\p})^{\T})& -  \alpha_2^{-1}(I_n\otimes\mathfrak{R}(\widehat{\p})^{\T})} ,
  \end{align*}}
  and
  {\small \begin{align*}
    \widehat{\mathbf{Y}}_2= \bmatrix{ \alpha_2^{-1}(\mathfrak{R}(\widehat{\u})^{\T}\otimes I_m) &-\alpha_2^{-1}(\mathfrak{I}(\widehat{\u})^{\T}\otimes I_m)& \alpha_3^{-1}(\mathfrak{R}(\widehat{\p})^{\T}\otimes I_m) & -\alpha_3^{-1}(\mathfrak{I}(\widehat{\p})^{\T}\otimes I_m)\\
      \alpha_2^{-1}(\mathfrak{I}(\widehat{\u})^{\T}\otimes I_m) &\alpha_2^{-1} (\mathfrak{R}(\widehat{\u})^{\T}\otimes I_m) &  \alpha_3^{-1}(\mathfrak{I}(\widehat{\p})^{\T}\otimes I_m) & \alpha_3^{-1} (\mathfrak{R}(\widehat{\p})^{\T}\otimes I_m)}.
  \end{align*}
  }
\end{corollary}
\proof As we are not preserving the sparsity structure of  $E$, $F$ and $G$, by setting $\Theta_E = {\bf 1}_{n\times n}$, $\Theta_F = {\bf 1}_{m\times n}$ and $\Theta_G = {\bf 1}_{m\times m}$, the proof proceeds accordingly. $\blacksquare$

In the following, we derive the structured BE for the GSPP \eqref{eq11} when $E\in \mathbb{HC}^{n\times n}$ and $G=0_{m\times m}.$
\begin{corollary}
   Assume that $E\in \mathbb{HC}^{n\times n},$ $F\in \C^{m\times n}$ and $G=0_{m\times m}$  and $\widehat{\bm{x}}=[\widehat{\bm{u}}^{\T},\widehat{\bm{p}}^{\T}]^{\T}$ is a computed solution of the GSPP \eqref{eq11}. Then, the structured BE $\bm{\xi}_{\tt sps}^{\mathcal{G}_1}(\widehat{\bm{u}},\widehat{\bm{p}})$ with preserving sparsity is given by  \begin{eqnarray}
      \bm{\xi}^{\mathcal{G}_1}_{\tt sps}(\widehat{\bm{u}},\widehat{\bm{p}})= \left\| \bmatrix{\widehat{\mathbf{X}}_1& \mathcal{I}_1\\ \mathbf{Z}& \mathcal{I}_2}^{\T}\left(\bmatrix{\widehat{\mathbf{X}}_1& \mathcal{I}_1\\ \mathbf{Z}& \mathcal{I}_2}\bmatrix{\widehat{\mathbf{X}}_1& \mathcal{I}_1\\ \mathbf{Z}& \mathcal{I}_2}^{\T}\right)^{-1}\bmatrix{\RR(Q)\\ \I(Q)\\ \RR(\widehat{R})\\ \I(\widehat{R})}\right\|_2,
    \end{eqnarray}
    where $\mathbf{Z}= [ {0}_{2m^2\times n^2}~~ \widehat{\mathbf{Z}}],$  $$\widehat{\mathbf{Z}}=\bmatrix{ \alpha_2^{-1}(\mathfrak{R}(\widehat{\u})^{\T}\otimes I_m)\Sigma_F &-\alpha_2^{-1}(\mathfrak{I}(\widehat{\u})^{\T}\otimes I_m)\Sigma_F\\
      \alpha_2^{-1}(\mathfrak{I}(\widehat{\u})^{\T}\otimes I_m)\Sigma_F &\alpha_2^{-1} (\mathfrak{R}(\widehat{\u})^{\T}\otimes I_m)\Sigma_F}$$
and $\widehat{R}=r-F\widehat{\u}.$
\end{corollary}
\proof The proof follows by taking $\alpha_3\rightarrow \infty$ and $G=0_{m\times m}$ in Theorem \ref{th1:case1}. $\blacksquare$

\vspace{1.2mm}
The structured BE for the GSPP \eqref{eq11} is analyzed in \citet{BE2020BING} under the assumption that the block matrices are real, i.e., \( E = E^{\T} \in \mathbb{R}^{n \times n} \), \( F \in \mathbb{R}^{m \times n} \), \( G \in \mathbb{R}^{m \times m} \), and the vectors \( q \in \mathbb{R}^n \), \( r \in \mathbb{R}^m \). From our results, we can also derive the structured BE for this case, which is given as follows.
\begin{corollary}\label{BE_R1}
     Assume that $E\in \mathbb{SR}^{n\times n},$ $F\in \R^{m\times n},$ $G\in \R^{m\times m},$ $q\in \R^n,$ $r\in \R^m,$  and $\widehat{\bm{x}}=[\widehat{\bm{u}}^{\T},\widehat{\bm{p}}^{\T}]^{\T}$ is a computed solution of the GSPP \eqref{eq11}. Then, the structured BE $\bm{\xi}^{\mathcal{G}_1}_{\tt sps}(\widehat{\bm{u}},\widehat{\bm{p}})$ with preserving sparsity is given by  
    \begin{eqnarray}
      \bm{\xi}_{\tt sps}^{\mathcal{G}_1}(\widehat{\bm{u}},\widehat{\bm{p}})= \left\| \bmatrix{\mathbf{X}^R_1& \widetilde{\mathcal{I}}_1\\ \mathbf{X}^R_2& \widetilde{\mathcal{I}}_2}^{\T}\left(\bmatrix{\mathbf{X}^R_1& \widetilde{\mathcal{I}}_1\\ \mathbf{X}^R_2& \widetilde{\mathcal{I}}_2}\bmatrix{\mathbf{X}^R_1& \widetilde{\mathcal{I}}_1\\ \mathbf{X}^R_2& \widetilde{\mathcal{I}}_2}^{\T}\right)^{-1}\bmatrix{Q\\ R}\right\|_2,
    \end{eqnarray}
    where $Q=q-E\u-F^{\T}\p$, $R=r-F\u-G\p,$ $\widetilde{\mathcal{I}}_1=\bmatrix{ -{\beta_1}^{-1} I_{n} & 0_{n\times m}}\in \R^{n\times (n+m)},$  $\widetilde{\mathcal{I}}_2=\bmatrix{ 0_{m \times n}&-{\beta_2}^{-1} I_{m}}\in \R^{m\times (n+m)},$
    \begin{align*}
  \mathbf{X}^R_1= \bmatrix{\alpha_1^{-1}(\widehat{\u}^{\T}\otimes I_n)N_1 & \alpha_2^{-1} (I_n\otimes\widehat{\p}^{\T})\Sigma_F &0_{n\times m^2}} 
  \end{align*}
  and
 \begin{align*}
    \mathbf{X}^R_2= \bmatrix{ 0_{m\times \frac{n(n+1)}{2}}&\alpha_2^{-1}(\widehat{\u}^{\T}\otimes I_m)\Sigma_F & \alpha_3^{-1}(\widehat{\p}^{\T}\otimes I_m)\Sigma_G }.
  \end{align*}
\end{corollary}
\proof Since  $E\in \mathbb{SR}^{n\times n},$ $F\in \R^{m\times n},$ $G\in \R^{m\times m},$ $q\in \R^n,$ $r\in \R^m,$ the computed solution $\widehat{\bm{x}}=[\widehat{\bm{u}}^{\T},\widehat{\bm{p}}^{\T}]^{\T}$ is real. Therefore, $\I(\widehat{\bm{u}})=0 $ and  $\I(\widehat{\bm{p}})=0.$ Furthermore, $\I(Q)=0$ and $\I(R)=0.$ Considering these values in Theorem \ref{th1:case1}, the desired structured BE is obtained. Hence, the proof is completed.
$\blacksquare$

By applying Remark \ref{remark1}, we can also obtain the optimal backward perturbation matrices for the structured BE.  The above result highlights the generalized nature of our proposed framework. Furthermore, our analysis ensures the preservation of the sparsity pattern of block matrices, unlike the results in \cite{BE2020BING}, which do not maintain the sparsity of the block matrices and provide the optimal backward perturbation matrices.
\subsection{Computation of the structured BE for case (\textit{ii})}
In this subsection, we derive a compact expression for the structured BE $\bm{\xi}_{\tt sps}^{\mathcal{G}_2} (\widehat{\bm{u}},\widehat{\bm{p}})$  by preserving the Hermitian structure of $G\in \mathbb{HC}^{m\times m}$ and preserving the sparsity of the coefficient matrix $\mathfrak{B}$.
Before starting the main theorem, we construct the following matrices.

\noindent Set 
$S_1:=\mathcal{J}^m_{S}\Phi_{G}\mathfrak{D}^{-1}_{S,m}\in \R^{m^2\times \frac{m(m+1)}{2}}, ~~S_2:=\mathcal{J}^m_{SK}\Psi_{G}\mathfrak{D}^{-1}_{SK,m}\in \R^{m^2\times \frac{m(m-1)}{2}},$
  $$\mathbf{K}_1:= [\widehat{\mathbf{K}}_1~~ {0}_{2n\times m^2}]~~\text{and}~~ \mathbf{K}_2:= [ {0}_{2m\times 2n^2}~~ \widehat{\mathbf{K}}_2],$$
 where
{\footnotesize\begin{align*}
    \widehat{\mathbf{K}}_1 = \bmatrix{\alpha_1^{-1}(\mathfrak{R}(\widehat{\u})^{\T}\otimes I_n)\Sigma_E &-\alpha_1^{-1}(\mathfrak{I}(\widehat{\u})^{\T}\otimes I_n)\Sigma_E& \alpha_2^{-1} (I_n\otimes\mathfrak{R}(\widehat{\p})^{\T})\Sigma_F & \alpha_2^{-1}(I_n\otimes\mathfrak{I}(\widehat{\p})^{\T})\Sigma_F\\
     \alpha_1^{-1} (\mathfrak{I}(\widehat{\u})^{\T}\otimes I_n)\Sigma_E &  \alpha_1^{-1}(\mathfrak{R}(\widehat{\u})^{\T}\otimes I_n)\Sigma_E &  \alpha_2^{-1}(I_n\otimes\mathfrak{I}(\widehat{\p})^{\T})\Sigma_F& -  \alpha_2^{-1}(I_n\otimes\mathfrak{R}(\widehat{\p})^{\T})\Sigma_F }
  \end{align*}}
  and
  {\footnotesize \begin{align*}
    \widehat{\mathbf{K}}_2 = \bmatrix{ \alpha_2^{-1}(\mathfrak{R}(\widehat{\u})^{\T}\otimes I_m)\Sigma_F &-\alpha_2^{-1}(\mathfrak{I}(\widehat{\u})^{\T}\otimes I_m)\Sigma_F& \alpha_3^{-1}(\mathfrak{R}(\widehat{\p})^{\T}\otimes I_m)S_1& -\alpha_3^{-1}(\mathfrak{I}(\widehat{\p})^{\T}\otimes I_m) S_2\\
       \alpha_2^{-1}(\mathfrak{I}(\widehat{\u})^{\T}\otimes I_m)\Sigma_F &\alpha_2^{-1} (\mathfrak{R}(\widehat{\u})^{\T}\otimes I_m)\Sigma_F&  \alpha_3^{-1}(\mathfrak{I}(\widehat{\p})^{\T}\otimes I_m)S_1& \alpha_3^{-1} (\mathfrak{R}(\widehat{\p})^{\T}\otimes I_m)S_2}.
  \end{align*}
  }
Denote \begin{align}\label{Eq:312}
   \D \mathcal{Y}:= \bmatrix{\alpha_1\vec(\mathfrak{R}(\D E\odot
  \Theta_E))\\ \alpha_1\vec_(\mathfrak{I}(\D E\odot
  \Theta_E))\\ \alpha_2\vec(\mathfrak{R}(\D F\odot
  \Theta_F))\\ \alpha_2\vec(\mathfrak{I}(\D F\odot
  \Theta_F))\\ \alpha_3\mathfrak{D}_{S, m}\vec_{S}(\mathfrak{R}(\D G\odot
  \Theta_G))\\ \alpha_3\mathfrak{D}_{SK}\vec_{SK, m}(\mathfrak{I}(\D G\odot
  \Theta_G))}\in \R^{\bm t}~~ \text{and} ~~ \D \mathcal{Z}:=\bmatrix{\beta_1\RR(\D q)\\ \beta_1\I(\D q)\\\beta_2 \RR(\D r)\\\beta_2 \I(\D r)}\in \R^{2(n+m)},
\end{align}
where $\bm{t}=2(n^2+nm)+m^2.$
\begin{theorem}\label{th1:case2}
  Assume that $E\in \mathbb{C}^{n\times n},$ $F\in \C^{m\times n}$ and $G\in \mathbb{HC}^{m\times m}$  and $\widehat{\bm{x}}=[\widehat{\bm{u}}^{\T},\widehat{\bm{p}}^{\T}]^{\T}$ is a computed solution of the GSPP \eqref{eq11}. Then, the structured BE $\bm{\xi}_{\tt sps}^{\mathcal{G}_2}(\widehat{\bm{u}},\widehat{\bm{p}})$ with preserving sparsity is given by  
    \begin{eqnarray}
      \bm{\xi}_{\tt sps}^{\mathcal{G}_2}(\widehat{\bm{u}},\widehat{\bm{p}})= \left\| \bmatrix{\mathbf{K}_1& \mathcal{I}_1\\ \mathbf{K}_2& \mathcal{I}_2}^{\T}\left(\bmatrix{\mathbf{K}_1& \mathcal{I}_1\\ \mathbf{K}_2& \mathcal{I}_2}\bmatrix{\mathbf{K}_1& \mathcal{I}_1\\ \mathbf{K}_2& \mathcal{I}_2}^{\T}\right)^{-1}\bmatrix{\RR(Q)\\ \I(Q)\\ \RR(R)\\ \I(R)}\right\|_2,
    \end{eqnarray}
    where $Q,R,$ $\mathcal{I}_1$ and $\mathcal{I}_2$ are defined as in Theorem \ref{th1:case1}.
\end{theorem}
\begin{proof}
    For the approximate solution $\widehat{\bm{x}}=[\widehat{\bm{u}}^{\T},\widehat{\bm{p}}^{\T}]^{\T},$ we are required to find the perturbations $\D E\in \C^{n\times n},$ $\D F\in \C^{m\times n},$ $\D G\in \mathbb{HC}^{m\times m}$ which retains the sparsity of $E, F$ and $G,$ respectively, and $\D q\in \C^n,$ $\D r\in \C^m$ so that \eqref{s3:eq25} holds. Thus, we replace $\D E, \D F$ and $\D G$  with $\D E\odot \Theta_E,$ $\D F\odot \Theta_F$ and $\D G\odot \Theta_G,$ respectively. Consequently, the following holds:
       \begin{align}\label{eq1:th2}
       \bmatrix{E+(\D E \odot \Theta_E) & (F+(\D F \odot \Theta_F))^{*}\\ F+(\D F \odot \Theta_F) & G+(\D G \odot \Theta_G)}\bmatrix{\widehat{\bm{u}}\\ \widehat{\bm{p}}}=\bmatrix{q+\D q\\ r+\D r},
  \end{align} 
  where $\D G\odot \Theta_G\in \mathbb{HC}^{m\times m}.$
Since, $\D G\odot \Theta_G\in \mathbb{HC}^{m\times m}$ implies that $\RR(\D G\odot \Theta_G)\in \mathbb{SR}^{m\times m}$ and $\I(\D G\odot \Theta_G)\in \mathbb{SKR}^{m\times m}.$ By applying Lemma \ref{lemm1:sec3} and following a similar approach as in Theorem \ref{th1:case1}, the proof is completed.
\end{proof}

\begin{remark}\label{remark2}
     In a similar fashion to Remark \ref{remark1}, we can easily construct the optimal backward perturbations ${\D E}^{\tt sps}_{\tt opt}$, ${\D F}^{\tt sps}_{\tt opt}$, ${\D G}^{\tt sps}_{\tt opt}$, ${\D q}_{\tt opt}$, and ${\D r}_{\tt opt}$. Moreover, by considering $\Theta_E={\bf 1}_{n\times n},$ $\Theta_F={\bf 1}_{m\times n}$ and $\Theta_G={\bf 1}_{m\times m},$ as in the Corollary \ref{coro1}, we can obtain the desired structured BE when the sparsity structure is not preserved.
\end{remark}

{Next, we present the structured BE for the GSPP \eqref{eq11}  by considering \( E \in \mathbb{R}^{n \times n} \), \( F \in \mathbb{R}^{m \times n} \), \( G=G^{\T} \in \mathbb{R}^{m \times m} \), and the vectors \( q \in \mathbb{R}^n \), \( r \in \mathbb{R}^m \). 
\begin{corollary}\label{BE_R2}
     Assume that $E\in \mathbb{R}^{n\times n},$ $F\in \R^{m\times n},$ $G\in \mathbb{SR}^{m\times m},$ $q\in \R^n,$ $r\in \R^m,$  and $\widehat{\bm{x}}=[\widehat{\bm{u}}^{\T},\widehat{\bm{p}}^{\T}]^{\T}$ is a computed solution of the GSPP \eqref{eq11}. Then, the structured BE $\bm{\xi}^{\mathcal{G}_2}_{\tt sps}(\widehat{\bm{u}},\widehat{\bm{p}})$ with preserving sparsity is given by  
    \begin{eqnarray}
      \bm{\xi}_{\tt sps}^{\mathcal{G}_2}(\widehat{\bm{u}},\widehat{\bm{p}})= \left\| \bmatrix{\mathbf{K}^R_1& \widetilde{\mathcal{I}}_1\\ \mathbf{K}^R_2& \widetilde{\mathcal{I}}_2}^{\T}\left(\bmatrix{\mathbf{K}^R_1& \widetilde{\mathcal{I}}_1\\ \mathbf{K}^R_2& \widetilde{\mathcal{I}}_2}\bmatrix{\mathbf{K}^R_1& \widetilde{\mathcal{I}}_1\\ \mathbf{K}^R_2& \widetilde{\mathcal{I}}_2}^{\T}\right)^{-1}\bmatrix{Q\\ R}\right\|_2,
    \end{eqnarray}
    where $Q=q-E\u-F^{\T}\p$, $R=r-F\u-G\p,$ $\widetilde{\mathcal{I}}_1=\bmatrix{ - {\beta_1}^{-1} I_{n} & 0_{n\times m}}\in \R^{n\times (n+m)},$  $\widetilde{\mathcal{I}}_2=\bmatrix{ 0_{m \times n}&- {\beta_2}^{-1} I_{m}}\in \R^{m\times (n+m)},$
    \begin{align*}
  \mathbf{K}^R_1= \bmatrix{\alpha_1^{-1}(\widehat{\u}^{\T}\otimes I_n)\Sigma_E & \alpha_2^{-1} (I_n\otimes\widehat{\p}^{\T})\Sigma_F &0_{n\times \frac{n(n+1)}{2}}} 
  \end{align*}
  and
 \begin{align*}
    \mathbf{K}^R_2= \bmatrix{ 0_{m\times n^2}&\alpha_2^{-1}(\widehat{\u}^{\T}\otimes I_m)\Sigma_F & \alpha_3^{-1}(\widehat{\p}^{\T}\otimes I_m)S_1 }.
  \end{align*}
\end{corollary}
\proof Given that \( E \in \mathbb{R}^{n \times n} \), \( F \in \mathbb{R}^{m \times n} \), \( G \in \mathbb{SR}^{m \times m} \), \( q \in \mathbb{R}^n \), and \( r \in \mathbb{R}^m \), then the computed solution \( \widehat{\bm{x}} = [\widehat{\bm{u}}^{\T}, \widehat{\bm{p}}^{\T}]^{\T} \) is real. As a result, \( \I(\widehat{\bm{u}}) = 0 \) and \( \I(\widehat{\bm{p}}) = 0 \). Additionally, we have \( \I(Q) = 0 \) and \( \I(R) = 0 \). Substituting these values into Theorem \ref{th1:case2}, we derive the desired structured BE, thus completing the proof.
$\blacksquare$

The structured BE for the case discussed in Corollary \ref{BE_R2} has also been studied in \cite{be2012LAA, be2017ma}.  However, our investigation also preserves the sparsity pattern of the block matrices and symmetric structure of the block $G\in \R^{m\times m}$, and provides optimal backward perturbation matrices.
}
\subsection{Computation of the structured BE for case (\textit{iii})} \label{sec:case3}
In this section, we derive the structured BE for the GSPP \eqref{eq11} when $E\in \mathbb{HC}^{n\times n},$ $F\neq H\in \C^{m\times n}$ and $G\in \mathbb{HC}^{m\times m}.$
Let $\bm{k}=n^2+4mn+m^2,$ then we construct the matrices $\mathbf{M}_1\in \R^{2n\times \bm{k}}$ and $\mathbf{M}_2\in \R^{2m\times \bm{k}}$  as follows:
 $$\mathbf{M}_1:= [\widehat{\mathbf{X}}_1~~ {0}_{2n\times (m^2+2mn)}]~~\text{and}~~ \mathbf{M}_2:= [ {0}_{2m\times (n^2+2mn)}~~ \widehat{\mathbf{M}}_2],$$
 where
  %
  {\footnotesize\begin{align*}
    \widehat{\mathbf{M}}_2=  \bmatrix{ \alpha_2^{-1}(\mathfrak{R}(\widehat{\u})^{\T}\otimes I_m)\Sigma_H &-\alpha_2^{-1}(\mathfrak{I}(\widehat{\u})^{\T}\otimes I_m)\Sigma_H& \alpha_3^{-1}(\mathfrak{R}(\widehat{\p})^{\T}\otimes I_m )S_1& -\alpha_3^{-1}(\mathfrak{I}(\widehat{\p})^{\T}\otimes I_m)S_2\\
       \alpha_2^{-1}(\mathfrak{I}(\widehat{\u})^{\T}\otimes I_m)\Sigma_H &\alpha_2^{-1} (\mathfrak{R}(\widehat{\u})^{\T}\otimes I_m)\Sigma_H &  \alpha_3^{-1}(\mathfrak{I}(\widehat{\p})^{\T}\otimes I_m)S_1& \alpha_3^{-1} (\mathfrak{R}(\widehat{\p}\otimes I_m)^{\T})S_2}.
  \end{align*}
  }
Moreover, we define the following two vectors:
\begin{align}\label{EQ:41}
   \D \mathcal{Y}:= \bmatrix{\alpha_1\mathfrak{D}_{S,n}\vec_{S}(\mathfrak{R}(\D E\odot
  \Theta_E))\\ \alpha_1\mathfrak{D}_{SK,n}\vec_{SK}(\mathfrak{I}(\D E\odot
  \Theta_E))\\ \alpha_2\vec(\mathfrak{R}(\D F\odot
  \Theta_F))\\ \alpha_2\vec(\mathfrak{I}(\D F\odot
  \Theta_F))\\\alpha_3\vec(\mathfrak{R}(\D H\odot
  \Theta_H))\\ \alpha_3\vec(\mathfrak{I}(\D H\odot
  \Theta_H))\\ \alpha_4\mathfrak{D}_{S,m}\vec_{S}(\mathfrak{R}(\D G\odot
  \Theta_G))\\ \alpha_4\mathfrak{D}_{S,m}\vec_{SK}(\mathfrak{I}(\D G\odot
  \Theta_G))}\in \R^{\bm k}~~ \text{and} ~~ \D \mathcal{Z}:=\bmatrix{\beta_1\RR(\D q)\\ \beta_1\I(\D q)\\\beta_2 \RR(\D r)\\\beta_2 \I(\D r)}\in \R^{2(n+m)}.
\end{align}

In the following theorem, we present the structured BE $\bm{\xi}_{\tt sps}^{\mathcal{G}_3}(\widehat{\bm{u}},\widehat{\bm{p}})$ by preserving the sparsity of the coefficient matrix $\mathfrak{B}.$
\begin{theorem}\label{th1:case3}
  Assume that $E\in \mathbb{HC}^{n\times n},$ $F, H\in \C^{m\times n},$ $G\in \mathbb{HC}^{m\times m}$  and $\widehat{\bm{x}}=[\widehat{\bm{u}}^{\T},\widehat{\bm{p}}^{\T}]^{\T}$ is an approximate solution of the GSPP \eqref{eq11}. Then, the structured BE $\bm{\xi}_{\tt sps}^{\mathcal{G}_3}(\widehat{\bm{u}},\widehat{\bm{p}})$ with preserving sparsity is given by  
    \begin{eqnarray}
      \bm{\xi}_{\tt sps}^{\mathcal{G}_3}(\widehat{\bm{u}},\widehat{\bm{p}})= \left\| \bmatrix{\mathbf{M}_1& \mathcal{I}_1\\ \mathbf{M}_2& \mathcal{I}_2}^{\T}\left(\bmatrix{\mathbf{M}_1& \mathcal{I}_1\\ \mathbf{M}_2& \mathcal{I}_2}\bmatrix{\mathbf{M}_1& \mathcal{I}_1\\ \mathbf{M}_2& \mathcal{I}_2}^{\T}\right)^{-1}\bmatrix{\RR(Q)\\ \I(Q)\\ \RR({R})\\ \I({R})}\right\|_2,
    \end{eqnarray}
    where $\mathcal{I}_1=\bmatrix{ -\beta_1^{-1} I_{2n} & 0_{2n\times 2m}}\in \R^{2n\times 2(n+m)},$  $\mathcal{I}_2=\bmatrix{ 0_{2m\times 2n}&-\beta_2^{-1} I_{2m}}\in \R^{2m\times 2(n+m)},$  $Q$ and $R$ defined as in Lemma \ref{lemm1:sec3}.  
\end{theorem}
\begin{proof}
    For the computed solution $\widehat{\bm{x}}=[\widehat{\bm{u}}^{\T},\widehat{\bm{p}}^{\T}]^{\T},$ we are required to find the perturbations $\D E\in \mathbb{HC}^{n\times n},$ $\D F, \D H\in \C^{m\times n},$ and $\D G\in \mathbb{HC}^{m\times m}$ which retains the sparsity of $E, F, H$ and $G,$ respectively, and $\D q\in \C^n,$ $\D r\in \C^m$ so that \eqref{s3:eq25} holds. Thus, we replace $E, F, H$ and $G$ with $\D E\odot \Theta_E,$ $\D F\odot \Theta_F$, $\D H\odot \Theta_H$ and $\D G\odot \Theta_G,$ respectively. Consequently, the following holds:
       \begin{align}\label{eq1:th3}
       \bmatrix{E+(\D E \odot \Theta_E) & (F+(\D F \odot \Theta_F))^{*}\\ H+(\D H \odot \Theta_H) & G+(\D G \odot \Theta_G)}\bmatrix{\widehat{\bm{u}}\\ \widehat{\bm{p}}}=\bmatrix{q+\D q\\ r+\D r},
  \end{align} 
  where $\D E\odot \Theta_E\in \mathbb{HC}^{n\times n}$ and $\D G\odot \Theta_G\in \mathbb{HC}^{m\times m}.$

\noindent By Lemma \ref{lemm1:sec3}, \eqref{eq1:th3} leads to
   \begin{align}
  \left\{\begin{array}{r}
  \mathfrak{R}(\D E\odot \Theta_E)\RR(\u)-\I(\D E\odot \Theta_E)\mathfrak{I}(\u)+ \mathfrak{R}(\D F\odot \Theta_F)^{\T}\RR(\p)\\
  + \I(\D F\odot \Theta_F)^{\T}\mathfrak{I}(\p) - \RR(\D q)=\RR(Q),\\
   \RR(\D E\odot \Theta_E)\I(\u)+\I(\D E\odot \Theta_E)\RR(\u)+\RR(\D F\odot \Theta_F)^{\T}\I(\p)\\
   -\I(\D F\odot \Theta_F)^{\T}\RR(\p)-\I(\D q)=\I(Q),
  \end{array}
  \right.
  \end{align}
  and
  \begin{align}
  \left\{\begin{array}{r}
  \mathfrak{R}(\D H\odot \Theta_H)\RR(\u)-\I(\D H\odot \Theta_H)\mathfrak{I}(\u)+ \mathfrak{R}(\D G\odot \Theta_G)\RR(\p)\\
  - \I(\D G\odot \Theta_G)\mathfrak{I}(\p) - \RR(\D r)=\RR(\widetilde{R}),\\
   \RR(\D H\odot \Theta_H)\I(\u)+\I(\D H\odot \Theta_H)\RR(\u)+\RR(\D G\odot \Theta_G)\I(\p)\\
   +\I(\D G\odot \Theta_G)\RR(\p)-\I(\D r)=\I(\widetilde{R}).
  \end{array}
  \right.
  \end{align}
Given that $\D E\odot \Theta_E\in \mathbb{HC}^{n\times n}$  and $\D G\odot \Theta_G\in \mathbb{HC}^{m\times m}.$ These give $\RR(\D E\odot \Theta_E)\in \mathbb{SR}^{n\times n}, $ $\I(\D E\odot \Theta_E)\in \mathbb{SKR}^{n\times n},$ $\RR(\D G\odot \Theta_G)\in \mathbb{SR}^{m\times m}$ and $\I(\D G\odot \Theta_G)\in \mathbb{SKR}^{m\times m}.$ Now,   using  a similar method to Theorem \ref{th1:case1}, we have
\begin{align}\label{th3:eq2}
     \bmatrix{\mathbf{M}_1& \mathcal{I}_1\\ \mathbf{M}_2& \mathcal{I}_2}\bmatrix{\D \mathcal{Y}\\ \D \mathcal{Z}}=\bmatrix{\RR(Q)\\ \I(Q)\\ \RR({R})\\ \I({R})}.
\end{align}
Observe that $ \bmatrix{\mathbf{M}_1& \mathcal{I}_1\\ \mathbf{M}_2& \mathcal{I}_2}$ is a full row rank matrix. Hence, the linear system \eqref{th3:eq2} is consistent and its minimum norm solution is given by
\begin{align}\label{eq3:th3}
      \bmatrix{\D \mathcal{Y}\\ \D \mathcal{Z}}_{\tt opt}=\bmatrix{\mathbf{M}_1& \mathcal{I}_1\\ \mathbf{M}_2& \mathcal{I}_2}^{\T}\left(\bmatrix{\mathbf{M}_1& \mathcal{I}_1\\ \mathbf{M}_2& \mathcal{I}_2}\bmatrix{\mathbf{M}_1& \mathcal{I}_1\\ \mathbf{M}_2& \mathcal{I}_2}^{\T}\right)^{-1}\bmatrix{\RR(Q)\\ \I(Q)\\ \RR(\widetilde{R})\\ \I(\widetilde{R})}. 
\end{align}
The remainder of the proof follows from \eqref{eq3:th3} and using a similar technique as the proof of Theorem \ref{th1:case1}.
\end{proof}

\begin{remark}
  In a similar manner to Remark \ref{remark1}, the optimal backward perturbations ${\D E}^{\tt sps}_{\tt opt}$, ${\D F}^{\tt sps}_{\tt opt}$, ${\D H}^{\tt sps}_{\tt opt}$, ${\D G}^{\tt sps}_{\tt opt}$, ${\D q}_{\tt opt}$, and ${\D r}_{\tt opt}$ can be easily derived. Furthermore, by taking $\Theta_E={\bf 1}_{n\times n}$, $\Theta_F={\bf 1}_{m\times n}=\Theta_H$ and $\Theta_G={\bf 1}_{m\times m}$, as in Corollary \ref{coro1}, we can derive the structured BE under the condition that the sparsity structure is not maintained.
\end{remark}
{The structured BE for the GSPP \eqref{eq11} has been examined in \citet{BE2020BING} under the assumption that the block matrices are real, specifically, \( E = E^{\T} \in \mathbb{R}^{n \times n} \), \( F, H \in \mathbb{R}^{m \times n} \), and \( G = G^{\T} \in \mathbb{R}^{m \times m} \), with the vectors \( q \in \mathbb{R}^n \) and \( r \in \mathbb{R}^m \). Based on our findings, we can also derive the structured BE for this scenario, which is presented below.
\begin{corollary}\label{BE_R3}
     Assume that $E\in \mathbb{SR}^{n\times n},$ $F,\,H\in \R^{m\times n},$ $G\in \mathbb{SR}^{m\times m},$ $q\in \R^n,$ $r\in \R^m,$  and $\widehat{\bm{x}}=[\widehat{\bm{u}}^{\T},\widehat{\bm{p}}^{\T}]^{\T}$ is a computed solution of the GSPP \eqref{eq11}. Then, the structured BE $\bm{\xi}^{\mathcal{G}_3}_{\tt sps}(\widehat{\bm{u}},\widehat{\bm{p}})$ with preserving sparsity is given by  
    \begin{eqnarray}
      \bm{\xi}_{\tt sps}^{\mathcal{G}_3}(\widehat{\bm{u}},\widehat{\bm{p}})= \left\| \bmatrix{\mathbf{M}^R_1& \widetilde{\mathcal{I}}_1\\ \mathbf{M}^R_2& \widetilde{\mathcal{I}}_2}^{\T}\left(\bmatrix{\mathbf{M}^R_1& \widetilde{\mathcal{I}}_1\\ \mathbf{M}^R_2& \widetilde{\mathcal{I}}_2}\bmatrix{\mathbf{M}^R_1& \widetilde{\mathcal{I}}_1\\ \mathbf{M}^R_2& \widetilde{\mathcal{I}}_2}^{\T}\right)^{-1}\bmatrix{Q\\ R}\right\|_2,
    \end{eqnarray}
    where $Q=q-E\u-F^{\T}\p$, $R=r-H\u-G\p,$ $\widetilde{\mathcal{I}}_1=\bmatrix{ - {\beta_1}^{-1} I_{n} & 0_{n\times m}}\in \R^{n\times (n+m)},$  $\widetilde{\mathcal{I}}_2=\bmatrix{ 0_{m \times n}&- {\beta_2}^{-1} I_{m}}\in \R^{m\times (n+m)},$
    \begin{align*}
  \mathbf{M}^R_1= \bmatrix{\alpha_1^{-1}(\widehat{\u}^{\T}\otimes I_n)N_1 & \alpha_2^{-1} (I_n\otimes\widehat{\p}^{\T})\Sigma_F &0_{n\times \frac{m(m+1)}{2}}} 
  \end{align*}
  and
 \begin{align*}
    \mathbf{M}^R_2= \bmatrix{ 0_{m\times \frac{n(n+1)}{2}}&\alpha_2^{-1}(\widehat{\u}^{\T}\otimes I_m)\Sigma_H & \alpha_3^{-1}(\widehat{\p}^{\T}\otimes I_m)S_1 }.
  \end{align*}
\end{corollary}
\proof Since \( E \in \mathbb{SR}^{n \times n} \), \( F,\, H \in \mathbb{R}^{m \times n} \), \( G \in \mathbb{SR}^{m \times m} \), \( q \in \mathbb{R}^n \), and \( r \in \mathbb{R}^m \), the computed solution \( \widehat{\bm{x}} = [\widehat{\bm{u}}^{\T}, \widehat{\bm{p}}^{\T}]^{\T} \) is real. Consequently, \( \I(\widehat{\bm{u}}) = 0 \) and \( \I(\widehat{\bm{p}}) = 0 \). Moreover, we have \( \I(Q) = 0 \) and \( \I(R) = 0 \). Substituting these values into Theorem \ref{th1:case3}, we obtain the desired structured BE, and hence the proof is completed.
$\blacksquare$

Our analysis preserves the sparsity pattern of block matrices, in contrast to the findings in \cite{BE2020BING}, which do not retain the sparsity of the block matrices.}
\section{Numerical examples}\label{sec:numerical}
In this section, we present a few numerical experiments to support our theoretical results.  We compare the structured BEs, both with and without retaining the sparsity structure, to the unstructured BE given in \eqref{UBE:exp}. Additionally, we construct the optimal perturbation matrices for achieving the structured BE. We report results illustrating the backward stability and strong backward stability of the numerical methods  applied to the GSPP.  {For all the examples, we take  $\alpha_1=\frac{1}{\|E\|_{\mathsf{F}}},$ $\alpha_2=\frac{1}{\|F\|_{\mathsf{F}}},$ $\alpha_3=\frac{1}{\|G\|_{\mathsf{F}}}$ (or $\alpha_3=\frac{1}{\|H\|_{\mathsf{F}}}$ and $\alpha_4=\frac{1}{\|G\|_{\mathsf{F}}}$), $\beta_1=\frac{1}{\|q\|_{2}},$ and $\beta_2=\frac{1}{\|r\|_{2}}.$}

All numerical experiments were performed using MATLAB R2024a on a system with an Intel(R) Core(TM) i7-10700 CPU running at 2.90 GHz, 16 GB of RAM.
\begin{exam}
    We consider the GSPP \eqref{eq11} with the following block matrices:
 {\small   \begin{align*}
        &E=\bmatrix{-0.7073&	-0.2258i&	-0.3326+0.4370i&	-0.3111-0.1089i&	-0.2558i\\
0.2258i&	1.6606&	1.0022&	-0.0749&	0.2357i\\
-0.3326-0.4370i	&1.0022	&0&	-1.5009&	-0.1383-0.0928i\\
-0.3111+0.1089i&	-0.0749&-1.5009	&0	&-0.1i\\
0.2558i&	-0.2357i&	-0.1383+0.0928i&	0.1i	&0},
 \end{align*}}
  \begin{align*}
& F=\bmatrix{-0.0753+1.3412i	&0	&0	&-1.3057&	0\\
-0.1974&	0&	2.9371	&0.3806i	&0 \\
0.2232+1.4354i&	0.7996i&	0.3985&	0&	1.6102\\
0.3862&	0.0097&	0&	1.6286i&	0.1291i}\\
&G=\bmatrix{1.5246-0.1337i&	0	&-0.6924&	-0.0408i\\
0	&-0.9025&	0&	0.0704\\
0&	-0.6885	+0.6028i&0.7823i&	1.2309\\
0.2146i	&0&0&	-0.2746},\\
&~~ q=\bmatrix{-0.8098	-0.3969i\\
-1.3853+	0.5947i\\
0.0909+	0.2202i\\
-0.2140	-0.7165i\\
0.1509+	0.0117i},~~\text{and}~~ r=\bmatrix{-2.3554	-0.9550i\\
0.6201	-0.7783i\\
0.3106+	1.5288i\\
-0.0908	-1.8683i}.
    \end{align*}
 Let \(\widehat{{\bm x}} = [\widehat{\u}^{\T}, \, \widehat{\p}^{\T}]^{\T}\) be a computed solution to the GSPP, where $$\widehat{\u} = \bmatrix{0.9249 + 1.6011i\\
-0.5210 + 0.2407i\\
0.0189 + 0.2151i\\
-1.5819 + 0.1480i\\
0.5443 + 1.2113i
}~~ \text{and} ~~\widehat{\p} = \bmatrix{-1.2670 - 1.2768i\\
-0.7997 + 0.4628i\\
0.4206 - 0.0082i\\
1.1641 - 1.0531i}$$
with residual \(\|\mathfrak{B}{\widehat{{\bm x}}} - \bm{f}\|_2 = 0.0012\). 

The unstructured BE ${\bm \xi}({\widehat{{\bm x}}})$ computed using the formula \eqref{UBE:exp} is $3.9295\times 10^{-05}.$
Given that \( E \) is Hermitian, we calculate the structured BE while preserving the sparsity using Theorem \ref{th1:case1} and without preserving the sparsity using Corollary \ref{coro1}, as listed below:
\begin{eqnarray}
   \bm{\xi}^{\mathcal{G}_1}_{\tt sps}(\widehat{\bm{u}},\widehat{\bm{p}})= 3.7327\times 10^{-04} ~~ \text{and} ~~ \bm{\xi}^{\mathcal{G}_1}(\widehat{\bm{u}},\widehat{\bm{p}})=3.2520\times 10^{-04}.
\end{eqnarray}
We observe that the structured BEs in both cases are only one order larger than the unstructured BE. Moreover, the structure-preserving optimal  {perturbation} matrices are given by 
{\footnotesize\begin{align*}
   & \D E^{\tt sps}_{\tt opt}\\
   &=10^{-5}\cdot \bmatrix{3.5894&	-3.0713-2.1972i&	-4.2943-2.2135i	&5.4466-3.3205i&	-2.1295+1.9597i\\
-3.0713+2.1972i	&1.2093	& 5.6451-1.0180i &	2.7731+3.7360i	&-2.5866+2.5464i\\
-4.2943+2.2135i&	0.5645+1.0180i&	0&	4.0067+2.8368i&	-3.5135+1.8975i\\
5.4466+3.3205i	&2.7731-3.7360i&	4.00675-2.8368i &	0&	4.0822-2.4857i\\
-2.1295-1.9597i	&-2.5866-2.5464i&	-3.5135	-1.8975i&4.0822+2.4857i&	0}, \\
&\D F^{\tt sps}_{\tt opt}\\
&=10^{-5}\cdot\bmatrix{-0.5850+ 5.2796i&	0&	0&	-8.0796 + 2.8769i&	0\\		
-1.5373 + 1.8646i&	0&	1.9608 - 4.3999i&	2.5302 + 7.2612i&	0	\\		
-9.6826- 8.4119i&	-3.4175 + 4.2353i&	-3.2788 + 1.0017i&	0&	-7.6773 - 2.9473i	\\		
-7.1701 - 5.4915i&	-2.2259 + 7.7132i&	0&	-4.9771 - 3.4259i&	-0.1621 + 3.5988i}, \\
& \D G^{\tt sps}_{\tt opt}= 10^{-5}\cdot \bmatrix{-1.5374 + 2.9958i&	-1.6890 - 0.3728i&	0.7551 - 0.2236i&	2.6735 + 1.2196i\\		
0&	0.0349 + 2.83169i&	0&	0.9484 - 4.7166i		\\
0&	-3.3959 + 4.2108i&	0.4254 - 2.4258i&	7.1382 - 5.7885i\\		
-1.2947 + 3.6675i&	0&	0&	3.2493 + 0.0981i}, \\
&\D q_{\tt opt}=10^{-5}\cdot \bmatrix{2.8974	-3.9154i\\
3.7448+	3.2211i\\
4.9361+	3.0047i\\
-1.1399	-2.9206i\\ 2.8144+	3.0114i
}~~\text{and}~~ \D r_{\tt opt}=10^{-5}\cdot\bmatrix{
-1.7842	+0.56641i\\
1.5676	+2.6335i\\
-0.8984+	5.7852i\\
-1.9543+	0.9252i}.
\end{align*}}
Observe that the computed perturbation matrices preserve the sparsity structure and $\D E_{\tt opt}^{\tt sps}$ is Hermitian.
Moreover, we have \begin{equation}
    (\mathfrak{B}+\D \mathfrak{B}^{\tt sps}_{\tt opt})\widehat{\bm{x}}=\bm{f}+\D \bm{f}_{\tt opt}, 
\end{equation}
where \begin{equation*}
    \D \mathfrak{B}^{\tt sps}_{\tt opt}=\bmatrix{\D E^{\tt sps}_{\tt opt}& {\D F^{\tt sps}_{\tt opt}}^{\T}\\ \D F^{\tt sps}_{\tt opt}& \D G^{\tt sps}_{\tt opt}}~~\text{and}~~ \D \bm{f}_{\tt opt}=\bmatrix{\D q_{\tt opt}\\\D r_{\tt opt}}.
\end{equation*}
\end{exam}

\begin{exam}
    To discuss the strong backward stability of the GMRES algorithm \cite{gmres} in solving the GSPP \eqref{eq11}, we conduct a comparative analysis between unstructured and structured BEs in this example. The input matrices of the GSPP \eqref{eq11} are constructed as follows:
    \begin{align*}
      &E=\RR(E)+i\,\I(E), ~ F={\tt sprandn}(m,n, 0.5)+i\, {\tt sprandn}(m,n, 0.5),\\
      &G={\tt sprandn}(m,m, 0.5)+i \,{\tt sprandn}(m,m, 0.5),~ q={\tt randn}(n,1)+i \,{\tt randn}(n,1),\\
      &\text{and}~~  r={\tt randn}(m,1)+i\, {\tt randn}(m,1) .
    \end{align*}
    where $\RR(E)=E_1+E_1^{*},$ $\I(E)=E_2+E_2^{*},$ $E_1={\tt sprandn}(n,n, 0.4)$~ \text{and} ~$E_2={\tt sprandn}(n,n, 0.4).$
Here, $ {\tt sprandn}(m,n, \mu )$ denotes the sparse random matrix of size $m\times n$ with $\mu mn$ nonzero entries and $\tt{randn(m,n)}$ denotes the random matrix of size $m\times n.$  The matrix dimensions are set as $n=3{\bf k}$ and $m=2{\bf k}.$ 

We use the GMRES algorithm with the zero initial guess vector and a stopping tolerance of \(10^{-8}\). Let \(\widehat{{\bm x}} = [\widehat{\u}^{\T},\, \widehat{\p}^{\T}]^{\T}\) represent the computed solution of the GSPP. In Figure \ref{fig1}, for \({\bf k} = 5:5:40\), we plot the unstructured BE \(\bm{\xi}(\widehat{{\bm x}})\) (labeled as  ``Unstructured BE''), structured BE with preserving sparsity \(\bm{\xi}_{\tt sps}^{\mathcal{G}_1}(\widehat{\u}, \widehat{\p})\) (labeled as  ``Structured BE with sparsity''), structured BE without preserving sparsity  {\(\bm{\xi}^{\mathcal{G}_1}(\widehat{\u}, \widehat{\p})\)} (labeled as  ``Structured BE'') and relative residual $\frac{\|\bm{f}-\mathfrak{B}\widehat{\bm{x}}\|_2}{\|\bm{f}\|_2}$ (labeled as  ``Residual'').

 \begin{figure}[]
					\centering
					\includegraphics[width=0.45\textwidth]{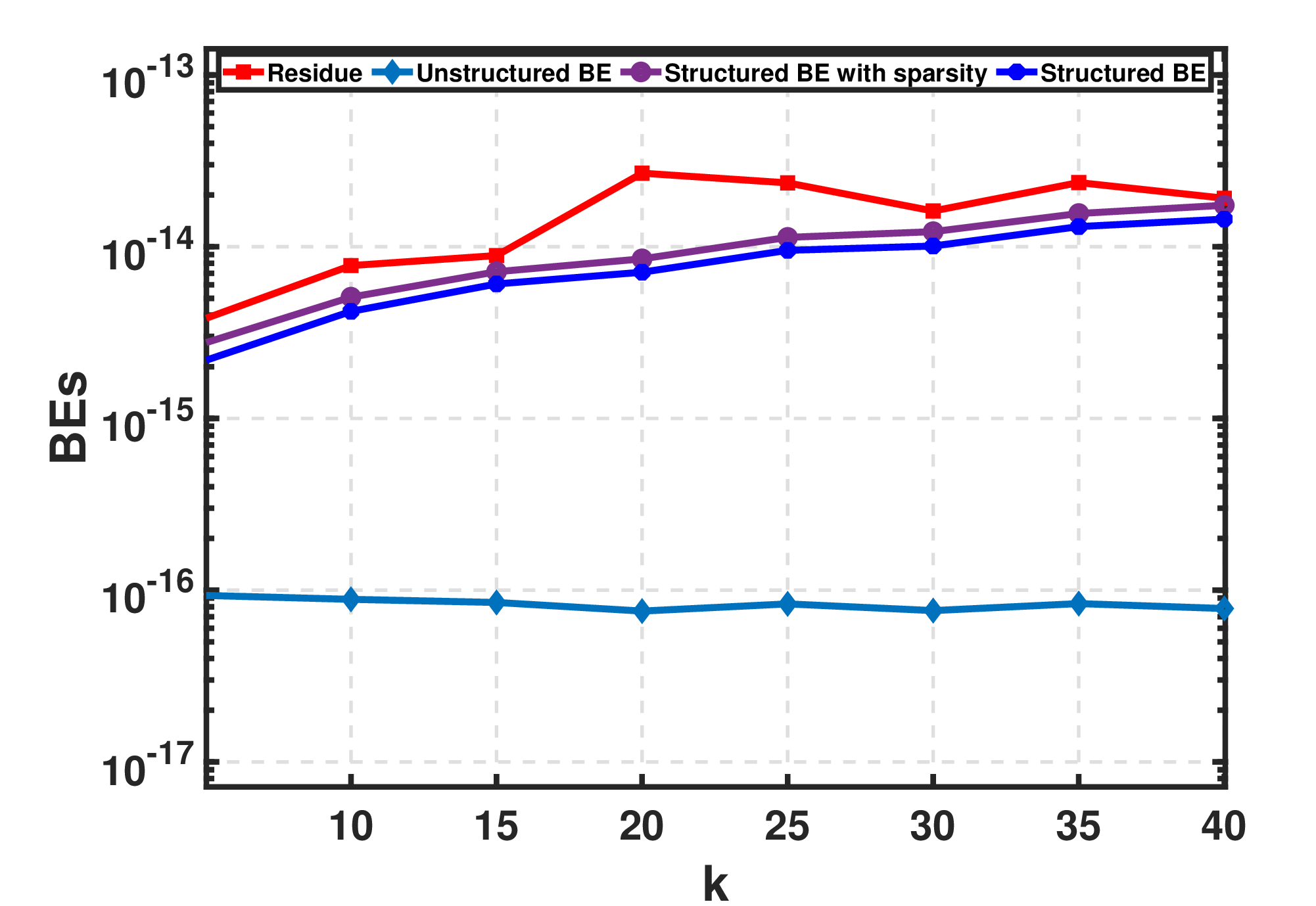}
		 \caption{ Relative residuals and BEs versus ${\bf k}$.}
		\label{fig1}
	\end{figure}
 From Figure \ref{fig1}, we observe that both structured backward errors (with and without preserving sparsity) are one to two orders of magnitude larger than the unstructured BE. Despite this difference, the structured BEs consistently remain below $\mathcal{O}(10^{-14})$, indicating that they are significantly small. Hence, our analysis shows that the GMRES algorithm exhibits both backward stability and strong backward stability for this given test problem.
\end{exam}

{\begin{exam}
We consider the GSPP \eqref{eq11} with the block matrices $E\in \mathbb{C}^{4\times 4}, \, F\in \mathbb{C}^{3\times 4}$ and $G\in  \mathbb{HC}^{3\times 3}$ and  the right hind side vectors $q\in  \mathbb{C}^{4},$ $r\in \mathbb{C}^{3}$   are given by
\begin{align*}
    &E=\bmatrix{0.01i & 10^{7}(1+i)&30(-1+i) &0\\ 100(1+i) & 0& 0& 10^5(-1+i)\\ 50(1+i) & 100(1+i)&0 &0\\ 0& 200(1-i) &10^5(1+i) & 0.01(1+i)},\\
    &F=\bmatrix{10^{-5}(1+i) & 10^7(1+i)& 0&0\\  10^8(1-i)& 10^{-5}(1+i) & -10^{-6}(1+i) &0\\ 0& 10^5(1+i)  &10^{-5}(1-i)  &10^6(1+i) },\\
    &G=\bmatrix{10^5& 0&100+0.01i\\ 0& 10^{-6} &0\\ 100-0.01i&0&-1},\, q=\bmatrix{10^4(1+i)\\ 10(1+i) \\ 0\\ 10^{-6}(1+i)},\, \text{and}\,~ r=\bmatrix{0.01(1-i)\\ 0\\ 0}.
\end{align*}
We use Gaussian elimination with partial pivoting (GEP) to solve the GSPP and the computed approximate solution is $\widehat{\bm{x}}=[\widehat{\bm{u}}^{\T},\,\widehat{\bm{p}}^{\T}]^{\T},$ where
\begin{align*}
    \widehat{\bm{u}}=\bmatrix{-0.0995-0.9904i\\
0.0005+0.0003i\\
-99.0353+9.9509i\\
0}\in \C^{4}~\text{and}~\widehat{\bm{p}}=\bmatrix{
-0.01-0.099i\\
0\\
0.9951+9.9035i}\in \C^{3}.
\end{align*}
We compute the unstructured BE $ \bm{\xi}(\widehat{\bm{x}})$ using the formula \eqref{UBE:exp} , structured BE by preserving sparsity \(\bm{\xi}_{\tt sps}^{\mathcal{G}_2}(\widehat{\u}, \widehat{\p})\) using Theorem \ref{th1:case2}, and structured BE without preserving sparsity \(\bm{\xi}^{\mathcal{G}_1}(\widehat{\u}, \widehat{\p})\) using Remark \ref{remark2}. The computed values are given by
\begin{align}
  \bm{\xi}(\widehat{\bm{x}}) =9.3151\times 10^{-20},~ \bm{\xi}_{\tt sps}^{\mathcal{G}_2}(\widehat{\u}, \widehat{\p})= 1.5494\times 10^{-9}~\text{and}~ \bm{\xi}^{\mathcal{G}_2}(\widehat{\u}, \widehat{\p})=1.4964\times 10^{-8}.
\end{align}
The unstructured backward error \( \bm{\xi}(\widehat{\bm{x}}) \) is on the order of \( \mathcal{O}(10^{-28}) \),  whereas the structured BEs $\bm{\xi}_{\tt sps}^{\mathcal{G}_2}(\widehat{\u}, \widehat{\p})$ and $\bm{\xi}^{\mathcal{G}_2}(\widehat{\u}, \widehat{\p})$ are significantly larger. Thus, the computed solution is an exact solution to a nearby unstructured linear system but not to a nearby structure-preserving GSPP.  Therefore, the GEP for solving this GSPP is backward stable but not strongly backward stable.
\end{exam}
\begin{exam}\label{exam4}
    In this example, we consider the GSPP \eqref{eq11} with the real block matrices 
		\begin{align*}
			& E= \bmatrix{I\otimes J+J\otimes I & 0\\  0&I\otimes J+ J\otimes I}\in \R^{2t^2\times2t^2}, \,\, F=\bmatrix{I\otimes X& X\otimes I}\in \R^{t^2\times 2t^2}, \\
		&\hspace{2cm}H= \bmatrix{Y\otimes X& X\otimes Y}\in \R^{t^2\times 2t^2}~\text{and}~~ G=0.
        \end{align*}
	  where $J=\frac{1}{(t+1)^2}\, \mathrm{tridiag}(-1,2,-1)\in \R^{t\times t},\quad X=\frac{1}{t+1}\,  \mathrm{tridiag}(0,1,-1)\in \R^{t\times t}$  and  $Y=\diag(1, t+1, \ldots, t^2-t+ 1)\in \R^{t\times t}.$ Here, $ \mathrm{tridiag}(a_1,a_2,a_3)\in \R^{t\times t}$ represents the tridiagonal matrix with the subdiagonal entry $a_1$, diagonal entry $a_2$,  and superdiagonal entry $a_3.$ The size of the coefficient matrix $\mathfrak{B}$ of GSPP is $ (n+m)=3t^2.$ The right-hand side vector $\bm{f}\in \R^{n+m}$ is taken such as the exact solution of the GSPP is $[1,1,\ldots,1]^{\T}\in \R^{n+m}.$

        \begin{table}[ht!]
		\centering
		\caption{Values of structured and unstructured BEs of the approximate solution obtained using GMRES for Example \ref{exam4}.}\label{tab1}
		\resizebox{9.5cm}{!}{
			\begin{tabular}{@{}ccccc@{}}
				\toprule
				$t$  & $\bm{\xi}(\widehat{\bm{x}})$ \cite{Rigal1967} & $\bm{\xi}_{\tt sps}^{\mathcal{G}_3}(\widehat{\bm{u}},\widehat{\bm{p}})^{\dagger}$ & $\bm{\eta}(\widehat{\bm{u}},\widehat{\bm{p}})$ \cite{BE2020BING}\\ [1ex]
				\midrule
$4$&$3.2543e-16$&$4.9110e-15$&$2.4997e-14$\\
$5$&$2.3891e-15$&$5.4614e-14$&$7.4194e-14$\\
$6$&$8.8212e-16$& $3.1719e-14$ &  $3.7464e-14$\\
$7$& $6.5032e-16$ & $3.1236e-14$ &  $8.6559e-14$\\
$8$& $6.0863e-16$ & $3.7896e-14$  & $4.0871e-14$ \\ [1ex]
				\bottomrule 
                \multicolumn{3}{l}{$\dagger$ indicates our obtained structured BE.}
		\end{tabular}}
	\end{table}

    We use the GMRES method to solve the GSPP. The initial guess vector is chosen $0\in \R^{n+m} $ and the stopping criterion is $\frac{\|\mathfrak{B}\bm{x}_k-\bm{f}\|_2}{\|\bm{f}\|_2}< 10^{-11},$ where $\bm{x}_k$ is the solution at $k^{th}$ iteration. We consider $t=4,5,\ldots, 8,$ and  compute the unstructured BE $\bm{\xi}(\widehat{\bm{x}})$ \cite{Rigal1967}, structured BE $\bm{\eta}(\widehat{\bm{u}},\widehat{\bm{p}})$ \cite{BE2020BING}, and structured BE by preserving sparsity using Corollary \ref{BE_R3} at the final iteration of the GMRES method.
    
 From Table \ref{tab1}, we observe that \( \bm{\eta}(\widehat{\bm{u}},\widehat{\bm{p}}) \) and \( \bm{\xi}_{\tt sps}^{\mathcal{G}_3}(\widehat{\bm{u}},\widehat{\bm{p}}) \) are on the order of \( \mathcal{O}(10^{-14}) \) for all values of \( t \), demonstrating the reliability of our structured BE formulae. Additionally, our structured BE formulation preserves the sparsity pattern of the coefficient matrix \( \mathfrak{B} \).  
On the other hand, the unstructured backward error \( \bm{\xi}(\widehat{\bm{x}}) \) remains around \( \mathcal{O}(10^{-16}) \) for all values of \( t \). This indicates that both the unstructured and structured BEs are significantly small, confirming that the GMRES method exhibits backward stable and strongly backward stable for solving this GSPP.
\end{exam}
}
\section{Conclusions}\label{sec:conclusion}
In this paper, we investigated the structured BEs for the GSPP under the constraint that the block matrices $E$ and $G$ preserve the Hermitian structure. Furthermore, we ensure that the perturbation matrices maintain the sparsity pattern of the coefficient matrix. We derive optimal perturbation matrices that simultaneously preserve the Hermitian structure and sparsity of the block matrices, identifying the closest perturbed structure-preserving GSPP. Thus, this ensures that the approximate solution corresponds to the exact solution of the perturbed GSPP. We conduct numerical experiments that validate the reliability and accuracy of our theoretical results. The derived structured BE formulas are utilized to assess the strong backward stability of the numerical methods to solve GSPPs.

\section*{Acknowledgments}
During the course of this work, Pinki Khatun was supported by a fellowship from the Council of Scientific \& Industrial Research (CSIR), New Delhi, India (File No. 09/1022(0098)/2020-EMR-I).

 \bibliography{Reference}
\bibliographystyle{abbrvnat}
\end{document}